\documentclass[11pt,reqno]{amsproc}
\usepackage{amsmath,latexsym,amssymb,verbatim,amsbsy,amsthm,mathrsfs,tikz,times}
\usepackage[margin=1in]{geometry}
\usepackage[shortlabels]{enumitem}
\usepackage{empheq}
\usepackage{float}

\title[Stability, well-posedness and blow-up criterion for the Incompressible Slice Model]{Stability, well-posedness and blow-up criterion for the Incompressible Slice Model}

\author[D.~Alonso-Or\'an]{Diego Alonso-Or\'an}
\address{Instituto de Ciencias Matem\'aticas CSIC-UAM-UC3M-UCM, 28049 Madrid, Spain.}
\email{diego.alonso@icmat.es}

\author[A.~Bethencourt de Le\'on]{Aythami Bethencourt de Le\'on}
\address{Department of Mathematics, Imperial College, London SW7 2AZ, UK. }
\email{ab1113@ic.ac.uk}

\usepackage{mathrsfs}
\usepackage{graphicx}
\usepackage{enumitem,linegoal}
\usepackage{calc}

\usepackage {color}

\usepackage[colorlinks=true, pdfstartview=FitV, linkcolor=blue, citecolor=blue, urlcolor=blue]{hyperref}
\theoremstyle{plain}
\newtheorem{theorem}{Theorem}[section]
\newtheorem{definition}[theorem]{Definition}
\newtheorem{lemma}[theorem]{Lemma}
\newtheorem{proposition}[theorem]{Proposition}
\newtheorem{corollary}[theorem]{Corollary}
\theoremstyle{definition}
\newtheorem{remark}[theorem]{Remark}

\def\tilde{\widetilde}
\numberwithin{equation}{section}

\renewcommand\hat{\widehat}

\DeclareMathOperator{\diff}{d\!}
\newcommand*\circled[1]{\tikz[baseline=(char.base)]{
            \node[shape=circle,draw,inner sep=2pt] (char) {#1};}}
\newcommand{\norm}[1]{\left\lVert#1\right\rVert}    
\newcommand\abs[1]{\left|#1\right|}    

\newcommand\restr[2]{{
  \left.\kern-\nulldelimiterspace 
  #1 
  \vphantom{\big|} 
  \right|_{#2} 
  }}

\begin{document}

\begin{abstract}
In atmospheric science, slice models are frequently used to study the behaviour of weather, and specifically the formation of atmospheric fronts, whose prediction is fundamental in meteorology. In 2013, Cotter and Holm introduced a new slice model, which they formulated using Hamilton's variational principle, modified for this purpose. In this paper, we show the local existence and uniqueness of strong solutions of the related ISM (Incompressible Slice Model). The ISM is a modified version of the Cotter-Holm Slice Model (CHSM) that we have obtained by adapting the Lagrangian function in Hamilton's principle for CHSM to the Euler-Boussinesq Eady incompressible case. Besides proving local existence and uniqueness, in this paper we also construct a blow-up criterion for the ISM, and study Arnold's stability around a restricted class of equilibrium solutions. These results establish the potential applicability of the ISM equations in physically meaningful situations.  \hfill \\\\ \today
\end{abstract}

\maketitle

\setcounter{tocdepth}{1}
\tableofcontents
\section{Introduction}
The Cotter-Holm Slice Model (CHSM) was introduced in \cite{CoHo2013} for oceanic and atmospheric fluid motions taking place in a vertical slice domain $\Omega \subset \mathbb{R}^2$, with smooth boundary $\partial \Omega$. The fluid motion in the vertical slice is coupled dynamically to the flow velocity transverse to the slice, which is assumed to vary linearly with distance normal to the slice. This assumption about the transverse flow through the vertical slice in the CHSM simplifies its mathematics while still capturing an important aspect of the 3D flow. The slice framework has been designed for the study of weather fronts; see, e.g., \cite{Ya}.
\\\\
In \cite{HosBre}, fronts were described mathematically and a general theory for studying fronts was developed. The necessary assumptions made are similar to the ones in CHSM. In general, slice models are used to study front formation with geostrophic balance in the cross-front direction. This assumption simplifies the analysis by formulating the dynamics in a two-dimensional slice, while still providing some realistic results.
\\\\
Front formation is directly related to baroclinic instability. Eady considered a classical model in 1949 (c.f. \cite{Eady}) in order to study the effects of baroclinic instability. Decades of observation have concluded that the most important source of synoptic scale variations in the atmosphere is due to the so called baroclinic instability. In \cite{BaWiHoFe}, this is linked to frontal systems and it is shown to trigger the formation of eddies in the North Sea.
\\\\
The Incompressible Slice Model (ISM) treated here comprises a particular case of the CHSM which is known as the Euler-Boussinesq Eady Slice Model. It should be emphasised that the ISM is potentially useful in numerical simulations of fronts. For instance, since the domain consists of a two-dimensional slice, computer simulations of ISM take much less time to run than a full three-dimensional model. Therefore, ISM is potentially useful in parameter studies for numerical weather predictions. There have been many studies on this kind of idealised models to predict and examine the formation and evolution of weather fronts (c.f. \cite{Ya}, \cite{NaHe}, \cite{Budd}, \cite{Vis}, \cite{Visram}). 
\\\\
The ISM resembles the standard 2D Boussinesq equations, which are commonly used to model large scale atmospheric and oceanic flows that are responsible for cold fronts and the jet stream \cite{Pedlosky}. The Boussinesq equations have been widely studied and considerable attention has been dedicated recently to their well-posedness and regularity  \cite{CanDiBenn}, \cite{Chae}, \cite{HouLi}. However, the fundamental question of whether their classical solutions blow up in finite time remains open. This problem is even discussed in Yudovich's ``eleven great problems of mathematical hydrodynamics" \cite{Yudo03}. Important progress in the global regularity problem has been made by Luo and Hou \cite{LuoHou1}, \cite{LuoHou2}, who have produced strong numerical evidence that smooth solutions of the 3D axisymmetric Euler equation system, which can be identified with the inviscid 2D Boussinesq equation, develop a singularity in finite time when the fluid domain has a solid boundary. Recently, Elgindi and Jeong \cite{ElgJeo} have shown finite-time singularity formation for strong solutions of the 2D Boussinesq system when the fluid domain is a sector of angle less than $\pi$.
\\\\
The goal of this paper is threefold: first, we provide a characterisation of a restricted class of equilibrium solutions of the ISM, and study the model's stability of solutions around it. This is performed by employing a modification of Arnold's techniques \cite{HoMaRaWe1985,Arnold78}. We also establish the local well-posedness of the ISM in Sobolev spaces, and provide a blow-up criterion for it. These are fundamental questions regarding the physical validity of these equations, which this paper answers in the affirmative.
\subsection*{Main results of the paper}
The ISM evolution equations for fluid velocity components $u_{S}(x,z,t):\Omega\subset\mathbb{R}^{2}\times\mathbb{R}^{+}\to\mathbb{R}^{2},$ scalar $u_{T}(x,z,t):\Omega\subset\mathbb{R}^{2}\times\mathbb{R}^{+}\to\mathbb{R}$ transverse to the slice, as well as the potential temperature $\theta_{S}(x,z,t):\Omega\subset\mathbb{R}^{2}\times\mathbb{R}^{+}\to\mathbb{R}$, are given by
\begin{subequations}\label{intialequations}
\begin{align}
\label{intro:Ism-eq:1} \partial_t u_S + u_S\cdot\nabla u_S 
-fu_T{\hat{x}}
&= -\nabla p 
 + \frac{g}{\theta_0}\theta_S\hat{z},  \\
\label{intro:Ism-eq:2} \partial_t u_T + u_S\cdot\nabla u_T 
+ fu_S\cdot\hat{{x}}
&=   -\frac{g}{\theta_0}zs,  \\
\label{intro:Ism-eq:3}\partial_t \theta_S + u_S\cdot\nabla\theta_S + u_T{s}  &= 0, \\
\label{intro:Ism-eq:4} \nabla\cdot u_S &=  0.
\end{align}
\end{subequations}
Here $g$ is the acceleration due to gravity, $\theta_{0}$ is the reference temperature, 	$f$ is the Coriolis parameter, which is assumed to be a constant, and $s$ is a constant which measures the variation of the potential temperature in the transverse direction. In these equations,  $\nabla$ denotes the 2D gradient in the slice, $p$ is the pressure obtained from incompressibility of the flow in the slice $(\nabla\cdot u_{S} =0)$, while $\hat{x}$ and $\hat{z}$ denote horizontal and vertical unit vectors in the slice. The flow is taken to be tangent to the boundary, so that
\begin{equation}\label{intro:ISM-bdy}
u_{S} \cdot n = 0 \text{ \ on  }  \partial \Omega.
\end{equation}
Here $n$ is the outward unit normal vector to the boundary $\partial{\Omega}$. The aim of the present paper is to prove the following theorems for the ISM equations \eqref{intro:Ism-eq:1}-\eqref{intro:Ism-eq:4} with boundary condition \eqref{intro:ISM-bdy}. The first of them deals with the formal and nonlinear stability of the equilibrium solutions of the Incompressible Slice Model:
\begin{theorem}[Equilibrium solutions of the ISM] \label{uno}
A restricted class of stationary solutions of the ISM equations  \eqref{intro:Ism-eq:1}-\eqref{intro:Ism-eq:4} with boundary condition \eqref{intro:ISM-bdy} is given by critical points of the generalised Hamiltonian $$H_\Phi = \int_\Omega \left\{ \frac{1}{2} (|u_S|^2 + u_T^2)  - \gamma_S \theta_S \right\} \diff{V} + \int_\Omega \Phi(q) \diff{V} + \sum_{i=0}^n a_i \int_{\partial \Omega_i } v_S \cdot \diff s.$$
These are given by the conditions
\begin{align*}
a_i &= \Phi'(\restr{q_e}{ \partial \Omega_i}), \quad \text{for \ } i=0, \ldots, n, \\
u_{Se} &=-\text{curl}(\Phi'(q_e) \hat{y}) s, \\
u_{Te} &= \text{curl}(\Phi'(q_e) \hat{y}) \cdot \nabla \theta_{Se},\\
\gamma_S &= \text{curl}(\Phi'(q_e) \hat{y}) \cdot (\nabla u_{Te} + f\hat{x}). 
\end{align*}

Here $\gamma_S = (g/\theta_0) z,$ $v_S = s u_S - (u_T + fx) \nabla \theta_S$ is the circulation velocity in the ISM, and $q={\rm curl} (v_S) \cdot \hat{y}$ is the potential vorticity.
Moreover, $\Phi$ can be written in terms of the Bernoulli function $K$ for the stationary solution as
\begin{equation*}
\Phi(\lambda) = \lambda  \left( \int_\lambda \frac{K(t)}{t^2} \diff t + C \right).
\end{equation*}
\end{theorem}

\begin{theorem}[Formal stability conditions for the ISM]
An equilibrium point of the ISM belonging to the restricted class specified in Theorem \ref{uno} is formally stable if
\begin{eqnarray}  \label{primacondicion}
\frac{(\hat{y} \times \nabla q_e) \cdot u_{Se}}{ |\nabla q_e|^2} >0.
\end{eqnarray}
\end{theorem}

\begin{remark}
This last result mimics the first Arnold's Theorem of formal stability for two-dimensional incompressible Euler \cite{Arnold78}, where the formal stability condition reads
$$\frac{(\hat{z} \times \nabla \omega_e) \cdot u_e}{|\nabla \omega_e|^2} >0.$$
Here $\omega_e = \text{curl}(u_e) \cdot \hat{z}.$ The extra term $q_R = - \text{curl}((u_T +fx) \nabla \theta_S)$ appearing in \eqref{primacondicion} is due to the introduction of a transverse velocity $u_{T}$ and a potential temperature $\theta_{S}$ . 
\end{remark}

\begin{theorem}[Nonlinear stability conditions for the ISM]
We can define a norm $Q$ on $\mathfrak{X}(\Omega) \circled{s} \mathcal{F}(\Omega) \times \wedge^2 (\Omega)$ such that an equilibrium point of the ISM belonging to the restricted class specified in Theorem \ref{uno} is nonlinearly stable with respect to $Q$ if
$$0 < \lambda_1 \leq \frac{(\hat{y} \times \nabla q_e) \cdot u_{Se}}{ |\nabla q_e|^2} \leq \lambda_2 < \infty.$$
\end{theorem}
Next, let us state the result which establishes the well-posedness of the system:
\begin{theorem}[Local well-posedness of the ISM]
For $s > 2$ integer and initial data $(u_{S}^{0},u_{T}^{0},\theta_{S}^{0})\in H^{s}_\star(\Omega) \times H^{s} (\Omega) \times H^{s}(\Omega)$, there exists a time $T=T(\norm{(u_{S}^{0},u_{T}^{0},\theta_{S}^{0})}_{H^{s}})>0$ such that the ISM equations \eqref{intro:Ism-eq:1}-\eqref{intro:Ism-eq:4} with boundary condition \eqref{intro:ISM-bdy} have a unique solution $(u_{S},u_{T},\theta_{S})$ in $C([0,T]; H^s_\star \times H^s \times H^s).$
\end{theorem}

In this paper we also prove a blow-up criterion for the Incompressible Slice Model, which reads as follows:
\begin{theorem}[Blow-up criterion for the ISM]   \label{blowupcriteria-intro}
Suppose that  $(u_{S}^{0},u_{T}^{0},\theta_{S}^{0})\in H^{s}_\star (\Omega)\times H^{s} (\Omega)\times H^{s} (\Omega)$ for $s>2$ an integer and that the solution $(u_{S},u_{T},\theta_{S})$ of equations \eqref{intro:Ism-eq:1}-\eqref{intro:Ism-eq:4} with boundary condition \eqref{intro:ISM-bdy} is of class $C([0,T];  H^{s}\times H^{s}\times H^{s})$. Then for $T^*< \infty,$ the following two statements are equivalent:
\begin{eqnarray}
&(i)& \quad E(t) < \infty, \quad \forall t < T^* \quad \mbox{ and } \quad \limsup_{t \to T^*} E(t) = \infty,
\label{eq:X(t)toinfty-intro} \\
&(ii)& \quad \int_0^{t} \norm{\nabla u_{S}(s)}_{L^{\infty}} \, \diff s < \infty, \quad \forall t < T^* \quad \mbox{ and } \quad \int_0^{T^*}  \norm{\nabla u_{S}(s)}_{L^{\infty}} \, \diff s = \infty,
\label{eq:utoinfty-intro}
\end{eqnarray}
where $E(t)=\norm{u_{S}}^{2}_{{H}^s}+\norm{u_{T}}^2_{{H}^s}+\norm{\theta_{S}}^{2}_{{H}^s}$. If such $T^*$ exists then $T^*$ is called the first-time blow-up and \eqref{eq:utoinfty-intro} is a blow-up criterion. 
\end{theorem}
\begin{remark}\rm
Theorem \ref{blowupcriteria-intro} could be used to validate whether the data from a given numerical simulation shows blow-up in finite time. Notice that the continuation criterion only depends on the velocity field $u_{S}$. This means we are able to control the scalar tracer velocity $u_{T}$ and the potential temperature $\theta_{S}$ globally in time by only obtaining a good control on $u_{S}$. 
\end{remark}
\begin{remark}\rm
Notice that we can recover the 2D incompressible Boussinesq system from the ISM \eqref{intro:Ism-eq:1}-\eqref{intro:Ism-eq:4}, by making $f,u_{T}=0$. The variable $u_{T}$ representing the transversal velocity to the slice, which naturally appears when deriving the ISM equations, gives rise to a more complex structure in the coupled system of equations. It is also worth mentioning that the conserved quantities available for the ISM are similar to the ones preserved in the 2D incompressible Boussinesq system.
\end{remark}

\subsection*{Structure of the paper} In Section \ref{2} we introduce some basic definitions and well-known lemmas about Sobolev spaces. We also include several preliminary results regarding Arnold's stability Theorem and the Kato-Lai Theorem for nonlinear evolution equations. At the end of Section \ref{2}, we establish the notation. Section \ref{1}  introduces the Cotter-Holm Slice Model by using its Lagrangian formulation, as carried out in \cite{CoHo2013}. In particular, we substitute the Euler-Boussinesq Lagrangian and derive the ISM equations, where we focus our interest throughout this manuscript. Since the Slice Model equations are Euler-Poincar\'e equations, they enjoy fundamental conservation laws (Kelvin's circulation, total energy, potential vorticity...). In Section \ref{nueva}, we characterise a class of equilibrium solutions of the ISM, and study formal and nonlinear stability around them by applying the Energy-Casimir method \cite{HoMaRaWe1985}. In Section \ref{3} we provide the first result of this paper, namely, we show the local well-posedness of the Incompressible Slice Model. In particular, we prove existence, uniqueness, and regularity of solutions, as well as their continuous dependence on the initial data. In Section \ref{4}, we construct a continuation type criterion which is well-known, for instance, for the Euler equation (see \cite{BKM84}). Finally, in Section \ref{5}, we propose some possible future research lines and comment on some open problems which are left to study.

\section{Preliminaries and basic notation}  \label{2}

Let $\Omega \subset \mathbb{R}^n$ be a bounded domain with smooth boundary $\partial \Omega$. For $\alpha\in\mathbb{N}^{n}$, $\alpha=(\alpha_{1},\alpha_{2},...,\alpha_{n}),$ and $f\in C^{\infty}(\Omega)$, we employ the multi-index notation
\[ D^\alpha f= D_{\alpha_1}D_{\alpha_2}\cdots D_{\alpha_n} f,\]
and denote $ |\alpha|=\displaystyle\sum_{i=1}^{n} |\alpha{_i}|.$ For any integer $s\in\mathbb{N}\cup\left\{ 0 \right\}$ and $p\in [1,\infty]$, we define the Sobolev norm
\[ \norm{f}_{W^{s,p}(\Omega)}= \sum_{\alpha\in\mathbb{N}^{n}: |\alpha|\leq s} \norm{D^{\alpha} f}_{L^p(\Omega)}. \]
Let us define the Sobolev space $W^{s,p}(\Omega)$ as the closure of $C^{\infty}$ functions with compact support  with respect to the norm $W^{s,p}(\Omega)$. When the spaces are $L^2$-based, they also turn out to be Hilbert and are denoted by $H^s(\Omega)=W^{s,2}(\Omega),$ with the interior product
\[ (f,g)_{H^{s}(\Omega)} = \sum_{\alpha\in\mathbb{N}^{n}: |\alpha|\leq s} \int_{\Omega}D^{\alpha} f D^{\alpha} g \ \diff V. \]
Let us introduce some important calculus inequalities \cite{BKM84}, \cite{KleMaj}:
\begin{lemma}\label{calculusineq}
\begin{minipage}[t]{\linegoal}
\begin{enumerate}[(i)]
\item \label{calculusineq:1} If $f,g\in H^{s}(\Omega)\cap C(\Omega)$, then 
\[\norm{fg}_{H^s} \leq C_{s,n} ( \norm{f}_{L^{\infty}}\norm{ D^{s} g }_{L^2} + \norm{D^{s} f}_{L^2} \norm{g}_{L^{\infty}}).\]
\item  \label{calculusineq:2} If $f\in H^{s}\cap C^{1}(\Omega)$ and $g\in H^{s-1}\cap C(\Omega)$, then for $|\alpha|\leq s,$
\[\norm{D^{\alpha}(fg)-fD^{\alpha}g}_{L^{2}} \leq C_{s,n}' ( \norm{f}_{W^{1,\infty}} \norm{g}_{H^{s-1}}+ \norm{f}_{H^s} \norm{g}_{L^{\infty}}).\]
\end{enumerate}
\end{minipage}
\end{lemma}
To estimate some boundary terms later on, we will invoke the so called Trace Theorem \cite{LioMag},  \cite{Aubin}.
\begin{theorem}\label{traceth}(Trace Theorem)
Let $u\in W^{s,p}(\Omega)$. Then there exist constants $C_{n,p,s}>0$ such that 
\[\norm{u}_{W^{s-\frac{1}{p},p}(\partial\Omega)}\leq C_{n,p,s}  \norm{u}_{W^{s,p}(\Omega)}. \]
\end{theorem}\vspace{0.2cm}
Now we introduce some functional spaces used throughout the paper. Let
\[ H_{\star}^{m}=\{ u\in H^{m}(\Omega) : \ \text{div}  \ u=0, \ u\cdot n=0 \ \text{on} \  \partial \Omega \} \]
for $m \geq 1,$ and
\[ H_{\star}^{0}=\{ u\in L^{2}(\Omega) : \ \text{div} \ u=0, \ u\cdot n=0 \ \text{on} \  \partial\Omega\}. \] 

Let us also mention the Helmholtz-Hodge decomposition Theorem and some properties of the Leray's projection operator.
\begin{lemma} \label{HelmHodge}
Assume that $\Omega \subset \mathbb{R}^n$ is a bounded domain with smooth boundary $\partial \Omega$ and $w$ is a vector field defined on $\Omega$. Then we can decompose $w$ in the form
\[ w=u+\nabla p,\]
where 
\[ \text{div} \ u=0, \ \ u\cdot n=0 \ \text{on } \partial\Omega, \ \ \int_{\Omega} u\cdot \nabla p \ \diff V=0. \]
\end{lemma}
The operator $\mathcal{P}:w\to u$ is called the Leray's projector and it has the following properties:
\begin{enumerate}[(i)]
\item $\text{div} \ (\mathcal{P}w)=0, \ \ (\mathcal{P}w) \cdot n=0 \ \text{on } \partial\Omega, \ \ \norm{\mathcal{P}w}_{L^2(\Omega)} \leq \norm{w}_{L^{2}(\Omega)}.$ \\
\item Set $\mathcal{Q}=1-\mathcal{P}$; this is, $\mathcal{Q}w=\nabla p$. Then if $w_{1}=u_{1}+\nabla p_{1},$ $w_{2}=u_{2}+\nabla p_{2}$, we have
\[ (\mathcal{P}w_{1},Qw_{2})_{L^2}=(u_{1},\nabla p_{2})_{L^2}=-(\text{div} \ u_{1},p_{2})_{L^2}= 0.\]
\end{enumerate} 

Among other things, we are interested in studying formal and nonlinear stability of the equilibrium solutions of the Incompressible Slice Model. We therefore introduce the method we use, which was developed in \cite{HoMaRaWe1985}. First of all, let us provide the definition of equilibrium solution:
\begin{definition}
Let $P$ be a Banach space of velocity fields $u$ and let $X:P \rightarrow P$ be an operator. This defines a dynamical system by 
\begin{equation}  \label{equilibrium}
\dot{u} = X(u).
\end{equation}
We say that a velocity profile $u_e$ is an equilibrium point of (\ref{equilibrium}) if $X(u_e)=0.$ 
\end{definition}

In a system like $(\ref{equilibrium}),$ one can study different types of stability:
 \begin{enumerate}[(i)]
 \item{Spectral stability. A system is said to be spectrally stable if the spectrum of its linearisation $DX(u_e)$ has no strictly positive real part.} \\
 \item{Linearised stability. A system is said to be linearised stable if its linearisation at the equilibrium point $u_e$ is stable. } \\
 \item{Formal stability. A system is said to be formally stable if there exists a conserved quantity such that its first variation vanishes at the equilibrium point $u_e,$ and whose second variation is definite (either positive of negative) at this point.}\\
 \item{Nonlinear stability. We say that an equilibrium solution $u_e$ is nonlinearly stable if there exists a norm $\norm{\cdot }$ and for every $\epsilon > 0$, there exists $\delta > 0$ such that if $\norm{u(0)-u_e}< \delta$ then $\norm{u(t) - u_e} < \epsilon,$ for $t > 0.$}
\end{enumerate}
\begin{remark}
It is important to study equilibrium solutions of physical systems and their stability. For instance, in the case of atmospheric models, these equilibrium solutions represent steady states of the atmospheric vector fields. Studying stability around steady states provides us with insight on whether some of these states can be destroyed in the short-term by small perturbations, or on the contrary, are expected to remain rather stable.
\end{remark}
We explain the Energy-Casimir algorithm presented in \cite{HoMaRaWe1985} for the study of formal and nonlinear stability of equilibrium solutions of a dynamical system. It consists of six steps:
\begin{enumerate}[(i)]
\item \label{stepa} Consider a system of the type (\ref{equilibrium}). The first step consists in finding an integral of motion $H$ for this system. Often, this equation can be expressed in terms of a Poisson bracket $\{\cdot, \cdot \},$ which we assume. \\
\item \label{stepb} Find a parametric family of constants of motion $C_\Phi$, where $\Phi$ belongs to some general class of functions. This is, these functions need to satisfy
$$\frac{D}{Dt} C_\Phi =0,$$
where $\dfrac{D}{Dt}$ stands for the total derivative. One way of doing this is to look for Casimirs of the Poisson bracket. These are functions $C$ such that $\left\{\cdot, C \right\} = 0.$ \\
\item \label{stepc} Construct a generalised conserved quantity
$$H_\Phi = H +  C_\Phi.$$
Impose $D H_\Phi(u_e) = 0$. This yields a condition on $ \Phi.$ \\
\item \label{stepd} Find quadratic forms $Q_1, Q_2$ on $P$ such that
\begin{align*}
Q_1(\Delta u) &\leq H(u_e +\Delta u)-H(u_e)-DH(u_e) \cdot \Delta u, \\ 
Q_2(\Delta u) &\leq C_\Phi (u_e +\Delta u)-C_\Phi (u_e)-DC_\Phi (u_e)\cdot \Delta u, 
\end{align*}
for $\Delta u \in P.$ Require that 
$$Q_1(\Delta u)+Q_2(\Delta u) > 0.$$
\item \label{stepe} Obtain the estimate
\begin{equation*}
Q_1(u(t)-u_e)+Q_2(u(t)-u_e) \leq H_\Phi (u(0))-H_\Phi(u_e),
\end{equation*}
often expressed as
\begin{equation*}
Q_1(\Delta u(t))+Q_2(\Delta u(t)) \leq H_\Phi (u(0))-H_\Phi(u_e).
\end{equation*}

\item \label{stepf} Define a norm on $P$ by
\[ 
\norm{u} = Q_1(u)+Q_2(u).
\]
\end{enumerate}

\begin{theorem}[Arnold Stability \cite{Arnold78}, \cite{HoMaRaWe1985}]
Assume steps \ref{stepa}-\ref{stepf} have been carried out. If $H_\Phi$ is continuous at $u_e$ on the norm  $\norm{\cdot}$, and solutions of \eqref{equilibrium} exist for all time, then $u_e$ is a nonlinearly stable equilibrium point.
\end{theorem}
\begin{remark}A sufficient condition for the continuity of $H_\Phi$ is the existence of positive constants $C_1, C_2$ such that
\begin{align*}
H(u_e +\Delta u)-H(u_e)-DH(u_e) \cdot \Delta u &\leq C_1\norm{\Delta u}^2, \\
 C_\Phi (u_e + \Delta u)-C_\Phi (u_e)-D C_\Phi(u_e)\cdot \Delta u &\leq C_2\norm{\Delta u}^2.
\end{align*}
\end{remark}\vspace{0.5cm}
Finally, we also wish to formulate an abstract theorem by Kato and Lai \cite{KatLai} that we will use to prove our local existence and uniqueness result for the ISM equations \eqref{intro:Ism-eq:1}-\eqref{intro:Ism-eq:4} with boundary condition \eqref{intro:ISM-bdy}. Before stating the theorem, we need to introduce some definitions first. We say that a family $\{V,H,X\}$ of three real separable Banach spaces is an admissible triplet if the following conditions are met:
\begin{itemize}
\item $V\subset H\subset X$, with the inclusions being dense and continuous. \\

\item $H$ is a Hilbert space with inner product $(\cdot,\cdot)_{H}$ and norm $\norm{\cdot}_{H}=(\cdot,\cdot)^{\frac{1}{2}}_{H}$. \\
\item There is a continuous, nondegenerate bilinear form on $V\times X$, denoted by $<\cdot,\cdot>$, such that
$$ <v,u> = (v,u)_{H}, \ \ v \in V, \ u \in H.$$
\end{itemize}
\begin{remark}
Recall that a bilinear form $<\cdot,\cdot>$ is continuous if
\[ |<v,u>| \leq  C\norm{v}_{V}\norm{u}_{X}, \ \ \text{for some constant \ } C>0, v\in V, \ \text{and \ }  u\in X. \]
For nondegeneracy we need
\[ <v,u>= 0 \ \ \text{for all } \  u\in X \ \text{implies} \ v=0, \]
\[<v,u>= 0 \ \ \text{for all } \ v\in V \ \text{implies} \ u=0.\]
\end{remark}
We say that $A:[0,T]\times H\to X$ is a sequentially weakly continuous map if $A(t_{n},v_{n})\rightharpoonup A(t,v)$ in $X$ whenever $t_{n}\rightarrow t$ and $v_{n}\rightharpoonup v$ in $H$. We denote by $C_{w}([0,T];H)$ the space of sequentially weakly continuous functions from $[0,T]$ into $H$, and by $C^{1}_{w}([0,T]; X)$ the space of functions $f\in W^{1,\infty}([0,T];X)$ such that $\frac{df}{dt}\in C_{w}([0,T]; X)$. With these notions established, the existence theorem by Kato and Lai reads:
\begin{theorem}\label{abstractevolution}(\cite{KatLai})
Consider the abstract nonlinear evolution equation
\begin{equation}\label{evolution}
\begin{cases}
u_{t}+A(t,u)=0, \\
u(0)=\phi,
\end{cases}
\end{equation}
where $A(t,u)$ is a nonlinear operator. Let $\{V,H,X\}$ be an admissible triplet. Let $A$ be a weakly continuous map on $[0,T_{0}]\times H$ into $X$ such that
\begin{eqnarray} \label{condicion}
<u,A(t,u)> \ \geq -\beta(\norm{u}^{2}_{H}), \  \ \text{for} \ t\in[0,T_{0}], \ u\in V,
\end{eqnarray}
where $\beta(r)\geq 0$ is a monotone increasing function of $r \geq 0.$ Then for any $\phi\in H,$ there exists $T\in(0,T_{0})$ such that \eqref{evolution} has a solution 
\[ u\in C_{w}([0,T];H)\cap C^{1}_{w}([0,T];X). \]
Moreover,
\[ \norm{u(t)}^{2}_{H} \leq \rho(t), \ t \in [0, T], \]
where $\rho(t)$ is a continuous increasing function on $[0,T]$.
\end{theorem}
\begin{remark}
If $A ([0,T_0] \times V) \subset H,$ we can rewrite \eqref{condicion} in a more convenient form, namely,
$$(u, A(t,u))_H \geq - \beta (\norm{u}_H^2), \quad t \in [0,T_0], \hspace{0.1cm} u \in V.$$
\end{remark}
\begin{remark}
$T$ and $\rho$ can be determined by solving the scalar differential equation 
\begin{equation}
\begin{cases}\label{weak}
\rho_{t}=2\beta(\rho), \\
\rho(0)= \norm{\phi}^{2}_{H},
\end{cases}
\end{equation}
where $T$ is any value that ensures \eqref{weak} has a solution $\rho(t)$ in $[0,T]$. It is important to mention that Theorem \ref{abstractevolution} does not guarantee uniqueness of solutions.
\end{remark}

\subsection*{Notation} We write $(\cdot,\cdot)_{L^p}$, $(\cdot,\cdot)_{H^m}$  for the scalar product in $L^{p}(\Omega)$ and $H^{m}(\Omega)$, respectively. $\norm{\cdot}_{H^{m}}$ will represent the norm in $H^{m}(\Omega)$. We shall use $\norm{\cdot}_{L^p}$ to denote the norm of a function in $L^{p}(\Omega)$. We write $|f|$ for the absolute value of $f$. We might use $L^{p},H^s$ throughout the paper meaning $L^{p}(\Omega)$ and $H^{s}(\Omega)$, respectively, when the domain $\Omega$ is implicit from the context. We use $\int$  to denote integration in space over $\Omega$. $a\lesssim b$ means there exists $C$ such that $a\leq  C b$, where $C$ is a positive universal constant that may depend on fixed parameters, constant quantities, and the domain itself. Note also that this constant might differ from line to line.

\section{The Incompressible Slice Model (ISM)}  \label{1}
In this section, we provide a review of the Cotter-Holm Slice Model. For this, we consider its Lagrangian particle map, which defines the evolution of a fluid configuration after a time $t,$ and we show that the possible configurations of this fluid system form a Lie group. The natural space for the velocities is the group's Lie algebra. On the Lie algebra, and for a choice of Lagrangian function, we can appy Hamilton's principle and derive evolution equations for the Slice Model dynamics.
\subsection{The Cotter-Holm Slice Model}
In the CHSM  \cite{CoHo2013}, one considers a dynamical system with Lagrangian evolution map of the form
\begin{align}
\phi(X,Y,Z,t) = (x(X,Z,t), y(X,Z,t) + Y, z(X,Z,t)),
\label{lagrangianmap}
\end{align}
for $(X,Z) \in \Omega \subset \mathbb{R}^2$ an open bounded domain, and $Y \in \mathbb{R}.$ Here, $\Omega$ is called the slice and $y$ is the transverse component to the slice. The change in time of the Eulerian coordinate paths $x(X,Z,t)$ and $z(X,Z,t)$ representing the motion of a fluid parcel in the vertical slice $\Omega,$ is assumed to depend only on its Lagrangian coordinates, or labels, $(X,Z)$ in the reference slice configuration. The change in the transverse Eulerian component $y(X,Y,Z,t)$ is taken to depend on the coordinates $X, Z,$ plus a linear variation in $Y.$  \\

\begin{figure}[H]\label{fig1}
\includegraphics[width=0.75\textwidth]{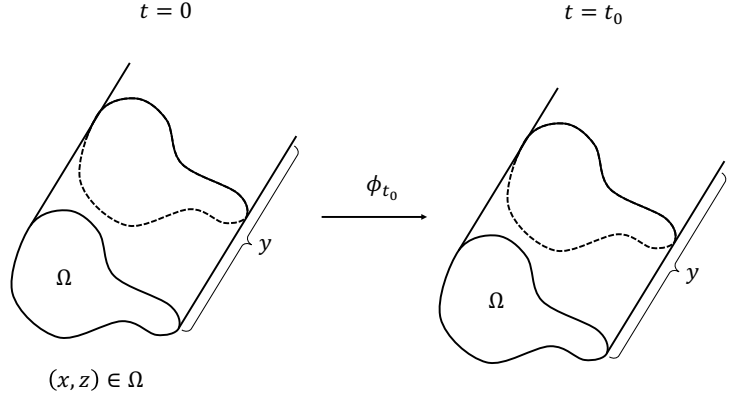}
\centering
\caption{The Lagrangian map $\phi$ explains how to move from a fluid configuration at time $t=0$ to a configuration at time $t=t_0.$}
\end{figure}

The set of Lagrangian maps of the type \eqref{lagrangianmap} can be modelled as $G$ = Diff$(\Omega)$ $\circled{s}$ $\mathcal{F}(\Omega),$ where Diff($\Omega$) denotes the group of diffeomorphisms in $\Omega,$ and $\mathcal{F}(\Omega)$ represents the group of differentiable real functions in $\Omega$. The symbol $\circled{s}$ denotes semi-direct product between two algebraic groups. $G$ can be endowed with a Lie group structure, with multiplication representing composition of Lagrangian maps \eqref{lagrangianmap}. Multiplication in the group is given by the formula
\begin{align}  
(\phi_1, f_1)*(\phi_2, f_2) = (\phi_2 \circ \phi_1, f_2 \circ \phi_1 + f_1),
\label{multiplication}
\end{align}
for $\phi_1, \phi_2 \in$ Diff($\Omega$), $f_1, f_2 \in \mathcal{F}(\Omega).$ This makes $(G, *)$ a Lie group. The operation $*$ turns out to be a right action.
\begin{remark}
Formula (\ref{multiplication}) represents the result of the composition of two Lagrangian maps of the type (\ref{lagrangianmap}), so the product on the Lie group describes the dynamics of the particles in the Slice Model. 
\end{remark}
If a Lie group $G$ represents the motions of a given physical system, the natural space for the velocities is its Lie algebra. The right-invariant Lie algebra of $G$ can be identified with the space $\mathfrak{g} = \mathfrak{X}(\Omega) \circled{s} \mathcal{F}(\Omega),$ where $\mathfrak{X}(\Omega)$ represents the set of right-invariant vector fields on $\Omega.$ This implies the velocity in this model has two components: the slice component $u_S \in \mathfrak{X}(\Omega),$ and the component in the $y$ direction (the transverse component), $u_T \in \mathcal{F}(\Omega).$ \\\\
The Lie bracket on $\mathfrak{g}$ for right-invariant vector fields (remember that the action on $G$ is right-invariant) is defined by the formula
$$[(u_S, u_T),(w_S, w_T)] = ([u_S, w_S], u_S \cdot \nabla w_T - w_S \cdot \nabla u_T),
\quad u_S, w_S \in \mathfrak{X} (\Omega), \quad u_T, w_T \in \mathcal{F}(\Omega).$$
                                                                                                                                                                                                                                                                                                                                                                                                                                                                                                                                                                                                                                                                                                                                                                                                                                                                                                                                                                                                                                                                                                                                                                                                                                                                                                                                                                                                                                                                                                                                                                                                                                                                                                                                                                                                                                                                                                                                                                                                                                                                                                                                                                                                                                                                                                                                                                                                                                                                                                                                                                                                                                                                                                                                                                                                                                                                                                                                                                                                                                                                                                                                                                                                                                                                                                                                                                                                                                                                                                                                                                                                                                                                                                                                                                                                                                                                                                                                                                                                                                                                                                                                                                                                                                                     Advected Eulerian quantities are defined to be the variables which are Lie transported by the Eulerian velocity. Conservation of mass comes from the right action by the inverse for advected quantities $a(t) = a_0g^{-1}(t),$ where $a_0$ denotes the advected quantity at time zero (in this case the mass density), and $g(t) \in G$ denotes the flow (in this case $\phi(t)).$ Conservation of mass in the CHSM reads
 \begin{equation*}                                                                                                                                                                                                                                                                                                                                                                                                                                                                                                                                                                                                                                                                                                                                                                                                                                                                                                                                                                                                                                                                                                                                                                                                                                                                                                                                                                                                                                                                                                                                                                                                                                                                                                                                                                                                                                                                                                                                                                                                                                                                                                                                                                                                                                                                                                                                                                                                                                                                                                                                                                                                                                                                                                                                                                                                                                                                                                                                                                                                                                                                                                                                                                                                                                                                                                                                                                                                                                                                                                                                                                                                                                                                                                                                                                                                                                                                                                                                                                                                                                                                                                                                                                                                                                                                                                                                                                                                                                                                                                                                                                                                                                                                                                                                                                                                                                                                                                                                                                                                                                                                                                                                                                                                                                                                                                                                                                                                                                                                                                                                                                                                                                                                                                                                                                                                                                                                                                                                                                                                                                                                                                                                                                                                                                                                                                                                                                                                                                                                                                                                                                                                                                                                                                                                                                                                                                                                                                                                                                                                                                                                                                                                                                                                                                                                                                                                                                                                                                                                                                                                                                                                                                                                                                                                                                                                                                                                                                                                                                                                                                                                                                                                                                                                                                                                                                                                                                                                                                                                                           \left( \frac{\partial}{\partial t} + \mathcal{L}_{(u_S, u_T)} \right) (D \diff^{ 3}x)                                                                                                                                                                                                                                                                                                                                                                                                                                                                                                                                                                                                                                                                                                                                                                                                                                                                                                                                                                                                                                                                                                                                                                                                                                                                                                                                                                                                                                                                                                                                                                                                                                                                                                                                                                                                                                                                                                                                                                                                                                                                                                                                                                                                                                                                                                                                                                                                                                                                                                                                                                                                                                                                                                                                                                                                                                                                                                                                                                                                                                                                                                                                                                                                                                                                                                                                                                                                                                                                                                                                                                                                                                                                                                                                                                                                                                                                                                                                                                                                                                                                                                                                                                                                                                           = \left( \frac{\partial D}{\partial t} + \nabla \cdot (u_S D) + \frac{\partial (u_T D)}{\partial y}  \right) \diff^{3} x =0,                                                                                                                                                                                                                                                                                                                                                                                                                                                                                                                                                                                                                                                                                                                                                                                                                                                                                                                                                                                                                                                                                                                                                                                                                                                                                                                                                                                                                                                                                                                                                                                                                                                                                                                                                                                                                                                                                                                                                                                                                                                                                                                                                                                                                                                                                                                                                                                                                                                                                                                                                                                                                                                                                                                                                                                                                                                                                                                                                                                                                                                                                                                                                                                                                                                                                                                                                                                                                                                                                                                                                                                                                                                                                                                                                                                                                                                                                                                                                                                                                                                                                                                                      
 \end{equation*}                                                                                                                                                                                                                                                                                                                                                                                                                                                                                                                                                                                                                                                                                                                                                                                                                                                                                                                                                                                                                                                                                                                                                                                                                                                                                                                                                                                                                                                                                                                                                                                                                                                                                                                                                                                                                                                                                                                                                                                                                                                                                                                                                                                                                                                                                                                                                                                                                                                                                                                                                                                                                                                                                                                                                                                                                                                                                                                                                                                                                                                                                                                                                                                                                                                                                                                                                                                                                                                                                                                                                                                                                                                                                                                                                                                                                                                                                                                                                                                                                                                                                                                                                                                                                                where $D$ denotes the mass density, and $\mathcal{L}_{(u_S, u_T)} (D \diff^{3} x)$ is the Lie derivative of the three form $D \diff^3 x =  D(x,y,z) \diff x \diff y \diff z$. Since $D$ and $u_T$ are assumed to be $y-$independent, conservation of mass can be reformulated as
\begin{eqnarray}  \label{mass}
                                                                                                                                                                                                                                                                                                                                                                                                                                                                                                                                                                                                                                                                                                                                                                                                                                                                                                                                                                                                                                                                                                                                                                                                                                                                                                                                                                                                                                                                                                                                                                                                                                                                                                                                                                                                                                                                                                                                                                                                                                                                                                                                                                                                                                                                                                                                                                                                                                                                                                                                                                                                                                                                                                                                                                                                                                                                                                                                                                                                                                                                                                                                                                                                                                                                                                                                                                                                                                                                                                                                                                                                                                                                                                                                                                                                                                                                                                                                                                                                                                                                                                                                                                                               \partial_t D + \nabla \cdot (u_S D)= 0.
                                                                                                                                                                                                                                                                                                                                                                                                                                                                                                                                                                                                                                                                                                                                                                                                                                                                                                                                                                                                                                                                                                                                                                                                                                                                                                                                                                                                                                                                                                                                                                                                                                                                                                                                                                                                                                                                                                                                                                                                                                                                                                                                                                                                                                                                                                                                                                                                                                                                                                                                                                                                                                                                                                                                                                                                                                                                                                                                                                                                                                                                                                                                                                                                                                                                                                                                                                                                                                                                                                                                                                                                                                                                                                                                                                                                                                                                                                                                                                                                                                                                                                                                                                                              \end{eqnarray}  
                                                                                                                                                                                                                                                                                                                                                                                                                                                                                                                                                                                                                                                                                                                                                                                                                                                                                                                                                                                                                                                                                                                                                                                                                                                                                                                                                                                                                                                                                                                                                                                                                                                                                                                                                                                                                                                                                                                                                                                                                                                                                                                                                                                                                                                                                                                                                                                                                                                                                                                                                                                                                                                                                                                                                                                                                                                                                                                                                                                                                                                                                                                                                                                                                                                                                                                                                                                                                                                                                                                                                                                                                                                                                                                                                                                                                                                                                                                                                                                                                                                                                                                                                                                             In this model, potential temperature is defined by
\begin{align}
\theta(x,y,z,t) = \theta_S(x,z,t) + (y-y_0) s,
\label{potentialt}
\end{align}
and the tracer equation for the potential temperature, implied by the right action of the inverse flow on advected quantities, becomes
\begin{eqnarray}  \label{temperature}
\partial_t \theta_S  + u_S \cdot \nabla \theta_S + u_T s= 0.
\end{eqnarray}

\begin{remark}
Note that \eqref{potentialt} is a very special way of defining potential temperature, since it is assumed it varies linearly on the $y$ direction. This proves to be useful for having circulation theorems (see Subsection \ref{cantidades}) which do not hold in the general case.
\end{remark}
Regarding the mathematical framework, $D$ is considered as an element in $\wedge^2 (\Omega),$ defined to be the space of two form-densities $\sum_{i=1}^m \alpha_i(x,z) \diff x \diff z,$ and $\theta_S$ is an element in $\mathcal{F}(\Omega)$ (a differentiable scalar function). Indeed, this is because equations $\eqref{mass}$ and $\eqref{temperature}$ can be rewritten respectively in Lie derivative notation as
$$(\partial_t + \mathcal{L}_{u_S}) (D \diff S) = 0,$$
and
$$(\partial_t + \mathcal{L}_{u_S}) \theta_S = -u_T s.$$
The CHSM considers the reduced Lagrangian function 
$$l[(u_S, u_T), (\theta_S,s), D]: (\mathfrak{X}(\Omega) \circled{s} \mathcal{F}(\Omega)) \circled{s} (( \mathcal{F}(\Omega) \times \mathbb{R}) \times \wedge^2 (\Omega)) \rightarrow \mathbb{R}.$$
As already explained, the tracers are advected by the Eulerian flow. The velocity vector field is right-invariant in Eulerian coordinates, so this implies that the Lagrangian function must be right-invariant, and therefore, reduction by symmetry can be performed. The equations of motion in the CHSM are the equations coming from the variational principle for this general Lagrangian function. This is, we consider the action functional
$$S[(u_S, u_T),(\theta_S,s), D] = \int_{0}^{T} l[(u_S, u_T), (\theta_S,s), D] \diff t,$$
where the last integral is taken over closed paths for the variables, vanishing at the endpoints. Apply Hamilton's principle and obtain
$$0 = \delta S = \delta \int_{0}^{T} l[(u_S, u_T), (\theta_S,s), D] \diff t,$$
from which one can derive the CHSM equations after integration by parts (c.f. \cite{CoHo2013}):
\begin{subequations}
\begin{align}
\label{Eady-EPSDeqns1} \partial_t \left( \frac{\delta l}{\delta u_S} \right) + \nabla \cdot \left( u_S \otimes \frac{\delta l}{\delta u_S} \right) + (\nabla u_S)^T \cdot \frac{\delta l}{\delta u_S} + \frac{\delta l}{\delta u_T} \nabla u_T &= D \nabla \frac{\delta l}{\delta D} - \frac{\delta l}{\delta \theta_S} \nabla \theta_S, \\
\label{Eady-EPSDeqns2} \partial_t  \left( \frac{\delta l}{\delta u_T} \right) + \nabla \cdot \left(u_S     \frac{\delta l}{\delta u_T}\right) &= - \frac{\delta l}{\delta \theta_S} s, \\
\label{Eady-EPSDeqns3} \partial_t \theta_S + u_S \cdot \nabla \theta_S + u_T s &= 0, \\
\label{Eady-EPSDeqns4} \nabla \cdot u_S &= 0.
\end{align}
\end{subequations}
These are the equations for a general (unspecified) Lagrangian function. In the next sections, we will substitute a particular Lagrangian, namely, the Lagrangian in the incompressible Euler-Boussinesq case, which has important physical significance.
\begin{figure}[H]\label{fig2}
\includegraphics[width=0.75\textwidth]{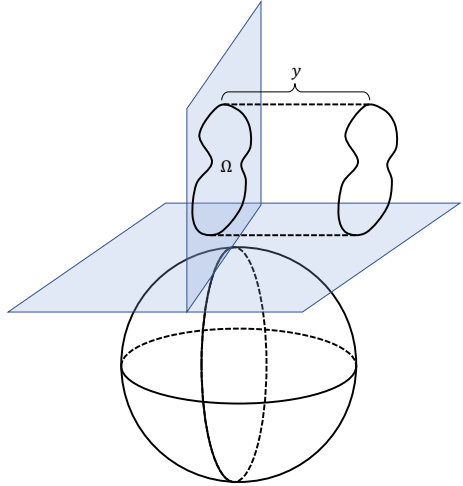}
\caption{Fluid motion in the vertical slice $\Omega$ coupled dynamically to the flow velocity transverse to the slice.}
\end{figure}

\subsection{The CHSM and front formation}
As a motivation for defining this model, we will say that slice models are very useful to study atmospheric processes where the dependence of the particles on one of the variables can be simply approximated. 

Following \cite{HosBre}, atmospheric fronts are generated by changing temperature gradients. More specifically, they are associated with discontinuities in velocity and potential temperature. There are many mechanisms which trigger frontogenesis, like:
\begin{enumerate}[(i)]
\item \label{horizontaldef} A horizontal deformation field.\\
\item \label{horizontalshear} Horizontal shearing motion.\\
\item \label{verticaldef} A vertical deformation field.\\
\item \label{diffvertical} Differential vertical motion.
\end{enumerate}
Mechanism \ref{horizontaldef} is the classical frontogenesis mechanism postulated by Bergeron (1928) \cite{Ber}. Mechanism \ref{horizontalshear} is crucial in the dynamics of frontal systems, and has been studied by Sawyer and Eliassen (1960s) \cite{Eli}. Mechanism \ref{diffvertical} can have either frontolytic or frontogenetic effects, and has been found responsible, for instance, for the lack of sharpness of surface fronts in the middle troposphere. Also, mechanisms \ref{horizontaldef}-\ref{horizontalshear} operate on the synoptic scale (they are large scale geostrophic processes), while \ref{verticaldef}-\ref{diffvertical} are dominant on the scale of the front (and are motions pertaining to the baroclinic flow which give rise to the rapid formation of a discontinuity). \\\\
Fronts form on the Earth when there is a strong temperature gradient on the North-South direction. In \cite{HosBre}, fronts are formulated mathematically. As an approximation, one considers geostrophic balance in the cross-front direction. After formulating the equations of motion and nondimensionalising, one seeks a solution of the type
\begin{align}\label{unaa} u &= -\alpha x + u'(x,z), \\ 
v &= \alpha y + v'(x,z), \\
w &= w(x,z), \\
\label{doos} \theta &= \theta(x,z), 
\end{align}
which is consistent with the approximation made in \cite{CoHo2013}. Note that $(u,v)$ in \eqref{unaa}-\eqref{doos} is represented by the 2D slice velocity $u_S$ in the CHSM, and $w$ is represented by $u_T$.

\subsection{The Incompressible Slice Model} 
For the incompressible Euler-Boussinesq Eady model in a smooth domain $(x,z)\in\Omega,$ the Lagrangian function is
\begin{eqnarray}  \label{lagrangiani}
l[u_S,u_T,D,\theta_S, p]  =\int_\Omega \left\{ \frac{D}{2} (|u_S|^2 + |u_T|^2) + Df u_T x + \frac{g}{\theta_0} D \theta_S z + p(1-D) \right\} \diff V,
\end{eqnarray}
where $g$ is the acceleration due to the gravity, $\theta_{0}$ is the reference temperature, $f$ is the Coriolis force parameter, which is assumed to be a constant, and $p$ is a multiplier which imposes the constraint $D=1,$ implying $\nabla \cdot u_S = 0$ (incompressibility).
\begin{remark}[Motivation for the Euler-Boussinesq Lagrangian]
Note that the integral in \eqref{lagrangiani} can be expressed as
\begin{equation*}
 \int_\Omega  KE(x,z) \diff V + \int_{\Omega} P_{f}E(x,z) \diff V + \int_\Omega IE(x,z) \diff V + \int_\Omega  p(1-D) \diff V =KE+P_{f}E+IE,
\end{equation*}
supplemented with constraint $D=1$. Above, $KE$ represents the kinetic energy, $P_{f}E$ the work done by the Coriolis force stored as potential energy in the fluid, and $IE$ the internal energy of the system.
\end{remark}
The ISM equations are the CHSM equations \eqref{Eady-EPSDeqns1}-\eqref{Eady-EPSDeqns4} for the Lagrangian function (\ref{lagrangiani}), which can be computed (\cite{CoHo2013}) to be
\begin{subequations}
\begin{align}
\label{Ism-eq:1} \partial_t u_S + u_S\cdot\nabla u_S 
-fu_T{\hat{x}}
&= -\nabla p 
 + \frac{g}{\theta_0}\theta_S\hat{z},  \\
\label{Ism-eq:2} \partial_t u_T + u_S\cdot\nabla u_T 
+ fu_S\cdot\hat{{x}}
&=   -\frac{g}{\theta_0} z {s},  \\
 \label{Ism-eq:3}\partial_t \theta_S + u_S\cdot\nabla\theta_S + u_T{s}  &= 0, \\
 \label{Ism-eq:4} \nabla\cdot u_S &=  0, 
\end{align}
\end{subequations}
supplemented with the boundary condition
\begin{equation} \label{Ism-eq:bdy}
 u_{S}\cdot n =0 \quad \text{on \ } \partial\Omega.
\end{equation}
Here $\hat{{x}}$ is the unit normal in the x-direction and $s$ is a constant which measures the variation of the potential temperature in the direction transverse to the slice (for further concreteness, see \cite{CoHo2013}).

\subsection{Conserved quantities in the ISM}  \label{cantidades}
The CHSM enjoys some conservation laws due to its variational character. Here, we state them in the particular case of the ISM, since it is the model we will be working with throughout the rest of this paper. The circulation velocity in the ISM is defined by
$$v_S = s u_S - (u_T + fx) \nabla \theta_S.$$
\begin{theorem}[Circulation conservation \cite{CoHo2013}]
The ISM \eqref{Ism-eq:1}-\eqref{Ism-eq:4} conserves the circulation of $v_S$ on loops $c(u_S)$ carried by $u_S$
\begin{align}
\frac{d}{dt}\oint_{c(u_s)} v_S\cdot \diff s 
= 	
\oint_{c(u_s)} 
\diff \pi =
0.
\label{EPEady-circons1}
\end{align}
\end{theorem}
\begin{corollary}[Conservation of potential vorticity \cite{CoHo2013}]
The ISM potential vorticity, defined by $q={\rm curl} (v_S) \cdot \hat{y}$ is conserved on fluid parcels,
\begin{equation}\label{generalisedvorti}
\frac{Dq}{Dt} = \partial_t q + u_S\cdot \nabla q = 0\,.
\end{equation}
\end{corollary}

\begin{remark}
Note that the definition of potential vorticity above differs from the usual definition of potential vorticity in fluid dynamics, since we have to take into account the transverse velocity $u_T,$ the Coriolis force $f,$ and the potential temperature $\theta_S.$ In other words, we are adding the Ertel potential vorticity $\nabla \theta_S \times \nabla (u_T + fx)$.
\end{remark}
The ISM has also the following integral conserved quantities:
\begin{theorem}[Conserved energy and enstrophy \cite{CoHo2013}]
The ISM \eqref{Ism-eq:1}-\eqref{Ism-eq:4} conserves energy and generalised enstrophy:
\begin{align}
h=\int_\Omega \left \{  \frac12 |u_S|^2 + \frac12 u_T^2 - \gamma_S \theta_S \right \} \diff V,
\quad\hbox{(Energy)} \quad
\,\label{EPEady-erg1}
\end{align}
\begin{align}
C_\Phi=\int_\Omega \Phi(q) \diff V,
\quad\hbox{(Generalised enstrophy)}\quad
\label{EPEady-erg2}
\end{align}
for any differentiable function $\Phi$ of the potential vorticity $q:={\rm curl}(v_S) \cdot \hat{y},$ and $\gamma_S = (g/\theta_0) z$.
\end{theorem}

\begin{remark}
Conserved quantities are fundamental to apply the Energy-Casimir algorithm (\cite{HoMaRaWe1985}) for the study of stability, and serve to guarantee integrability properties of the system.
\end{remark}

\section{Characterisation of ISM equilibrium solutions via Energy-Casimir}\label{nueva}
To study formal and nonlinear stability via the Energy-Casimir method \cite{HoMaRaWe1985} (see Section \ref{2} for a description of the algorithm), one constructs a generalised conserved quantity of the type
$$H_\Phi = h + C_\Phi .  $$
\begin{theorem} \label{equilibrio}
Critical points of $H_\Phi=h+C_\Phi$ are Eady equilibrium solutions of \eqref{Eady-EPSDeqns1}-\eqref{Eady-EPSDeqns4}.
\end{theorem}
\begin{proof}
This is a general property of the Energy-Casimir algorithm. For the proof, see the appendix in \cite{HoMaRaWe1985}.
\end{proof}
We present the first result of this paper, which consists in a characterisation of a class of equilibrium solutions of the ISM model.
\begin{theorem}[Equilibrium solutions of the ISM]\label{th:equilibria}
A class of stationary solutions of the ISM equations  \eqref{Ism-eq:1}-\eqref{Ism-eq:4} with boundary condition \eqref{Ism-eq:bdy} is given by critical points of the generalised Hamiltonian $$H_\Phi = \int_\Omega \left\{ \frac{1}{2} (|u_S|^2 + u_T^2)  - \gamma_S \theta_S \right\} \diff{V} + \int_\Omega \Phi(q) \diff{V} + \sum_{i=0}^n a_i \int_{\partial \Omega_i } v_S \cdot \diff s.$$
These are given by the conditions
\begin{subequations}\label{intro:conditions:equilibrium}
\begin{align}
a_i &= \Phi'(\restr{q_e}{ \partial \Omega_i}), \quad \text{for \ } i=0, \ldots, n, \label{intro:conditions:equi:1}\\
u_{Se} &=-\text{curl}(\Phi'(q_e) \hat{y}) s, \label{intro:conditions:equi:2}\\
u_{Te} &= \text{curl}(\Phi'(q_e) \hat{y}) \cdot \nabla \theta_{Se}, \label{intro:conditions:equi:3}\\
\gamma_S &= \text{curl}(\Phi'(q_e) \hat{y}) \cdot (\nabla u_{Te} + f\hat{x}). \label{intro:conditions:equi:4}
\end{align}
\end{subequations}
Here $\gamma_S = (g/\theta_0) z,$ $v_S = s u_S - (u_T + fx) \nabla \theta_S$ is the circulation velocity in the ISM, and $q={\rm curl} (v_S) \cdot \hat{y}$ is the potential vorticity. $\Omega_i,$ $i=0, \ldots, n,$ represent the different connected components of the domain $\Omega.$\\\\ 
Moreover, $\Phi$ can be written in terms of the Bernoulli function $K$ for the stationary solution as
\begin{equation*}
\Phi(\lambda) = \lambda \left( \int_\lambda \frac{K(t)}{t^2} \diff t + C \right).
\end{equation*}
\end{theorem}

\begin{proof}
We construct a generalised conserved quantity by
$$H_\Phi = h + C_\Phi = \int_\Omega \left\{ \frac{1}{2} (|u_S|^2 + u_T^2)  - \gamma_S \theta_S \right\} \diff{V} + \int_\Omega \Phi(q) \diff{V} + \sum_{i=0}^n a_i \int_{\partial \Omega_i } v_S \cdot \diff s.$$
Taking the first variation, one obtains
\begin{align*}
\delta H_\Phi(\delta u_S, \delta u_T, \delta \theta_S) &= \int_\Omega  (u_S \cdot \delta u_S + u_T \delta u_T  - \gamma_S \delta \theta_S ) \diff{V} \\
&+ \int_\Omega \Phi'(q) \delta q \diff{V} + \sum_{i=0}^n a_i \int_{\partial \Omega_i } \delta v_S \cdot \diff s:= I_{1}+I_{2}+I_{3}.
\end{align*}
We have that
\begin{align}
I_{2}+I_{3}&= \int_\Omega \Phi'(q) \text{curl}(\delta v_S) \cdot \hat{y} \diff{V} + \sum_{i=0}^n a_i \int_{\partial \Omega_i } \delta v_S \cdot \diff s \nonumber \\
&=\int_\Omega \text{curl}(\Phi'(q) \hat{y}) \cdot \delta v_S \diff{V}- \int_\Omega \text{div}(\Phi'(q) \hat{y} \times \delta v_S) \diff{V} + \sum_{i=0}^n a_i \int_{\partial \Omega_i } \delta v_S \cdot \diff s \nonumber \\
 \label{firstvar:1}&=\int_\Omega \text{curl}(\Phi'(q) \hat{y}) \cdot \delta v_S \diff{V} + \sum_{i=0}^n (a_i - \Phi'(q)_{|\partial \Omega_i}) \int_{\partial \Omega_i} \delta v_S \cdot \diff s, 
\end{align}
so we obtain $a_i = \Phi'(q_e | \partial \Omega_i),$ $i=1, \ldots, n.$ Here we have used the well-known calculus formula \[\text{div}(A \times B) = B \cdot \text{curl } A - A \cdot \text{curl } B, \quad \text{with } A,B \ \text{vector fields,}\] and the Divergence Theorem.
 We can rewrite the first term in \eqref{firstvar:1} as
\begin{align}
\int_\Omega \text{curl}(\Phi'(q) \hat{y}) \cdot \delta v_S \diff{V} &= \int_\Omega \text{curl}(\Phi'(q) \hat{y}) s \cdot \delta u_S \diff{V} - \int_\Omega \text{curl}(\Phi'(q) \hat{y}) \cdot \nabla \theta_S \delta u_T \diff{V} \nonumber \\
 \label{firstvar:2}&- \int_\Omega \text{curl}(\Phi'(q) \hat{y}) (u_T + f x) \cdot \nabla \delta \theta_S \diff{V}.
\end{align}
Notice that the last term in \eqref{firstvar:2}
\begin{align*}
 \int_\Omega \text{curl}(\Phi'(q) \hat{y}) (u_T + f x) \cdot \nabla \delta \theta_S \diff{V} = - \int_\Omega \text{curl}(\Phi'(q) \hat{y}) \cdot (\nabla u_T + f\hat{x}) \delta \theta_S \diff{V},
 \end{align*}
where we have taken into account that $\text{div } \text{curl } g = 0,$ for any smooth vector field $g,$ and we have needed the condition 
\begin{equation}\label{assumptionboundary}
\int_\Omega \text{div}(\text{curl}(\Phi'(q) \hat{y})(u_T + fx) \delta \theta_S) \diff{V} = \int_{\partial \Omega} (\text{curl}(\Phi'(q) \hat{y})(u_T + fx) \delta \theta_S) \cdot n \diff{s} =0.
\end{equation}
It is easily checked that \eqref{assumptionboundary} is guaranteed if $\text{curl}(\Phi'(q_e) \hat{y}) \cdot n =0$ at the boundary, which holds since
\[s \text{ curl}(\Phi'(q_e) \hat{y}) = -u_{Se},\]
and due to the boundary condition \eqref{Ism-eq:bdy}. Collecting all our previous computations we have that
\begin{align*}
\delta H_\Phi(\delta u_S, \delta u_T, \delta \theta_S) &= \int_\Omega  (u_S \cdot \delta u_S + u_T \delta u_T  - \gamma_S \delta \theta_S ) \diff{V} + \int_\Omega \text{curl}(\Phi'(q) \hat{y}) \cdot \delta v_S \diff{V} \\
&= \int_\Omega  (u_S \cdot \delta u_S + u_T \delta u_T  - \gamma_S \delta \theta_S ) \diff{V} + \int_\Omega \text{curl}(\Phi'(q) \hat{y}) s \cdot \delta u_S \diff{V} \\
&- \int_\Omega \text{curl}(\Phi'(q) \hat{y}) \cdot \nabla \theta_S \delta u_T \diff{V} +   \int_\Omega \text{curl}(\Phi'(q) \hat{y}) \cdot (\nabla u_T + f\hat{x}) \delta \theta_S \diff{V},
\end{align*}
obtaining the first part of the theorem. 
\begin{remark}
Note that the equations obtained for the variables after imposing the first variation of $H_\Phi$ to vanish satisfy the equilibrium conditions for the solutions of the ISM \eqref{Ism-eq:1}-\eqref{Ism-eq:4}, as explained in Theorem \ref{equilibrio}.  
\end{remark}
Let us now show how to rewrite the ISM equations in curl form at equilibrium, which is useful if we want to characterise its equilibrium solutions in terms of a Bernoulli function. We define $\omega_S = \text{curl}(u_S)$. One can express equation \eqref{Ism-eq:1} as
$$\frac{\partial u_S}{\partial t}  = - \omega_S \times u_S  - \nabla (p + |u_S|^2/2) + f u_T \hat{x}  + \frac{g}{\theta_0} \theta_S \hat{z}.$$\
Collecting all the gradient terms on the right-hand side, we have
\begin{align}
\frac{\partial u_S}{\partial t}  &= - \omega_S \times u_S  - \nabla (p + |u_S|^2/2) + f u_T \nabla x  + \frac{g}{\theta_0} \theta_S \nabla z \nonumber \\
\label{importante1}
&=- \omega_S \times u_S  - \nabla \left(p + |u_S|^2/2- f u_T x  -  \frac{g}{\theta_0} \theta_S z \right) - fx \nabla u_T  -\frac{g}{\theta_0} z\nabla  \theta_S.
\end{align}
We calculate
\begin{equation*} 
- (\nabla \theta_S \times \nabla (u_T + fx)) \times u_S = u_S \cdot \nabla (u_T + fx) \nabla \theta_S - (u_S \cdot \nabla \theta_S) \nabla(u_T +fx).
\end{equation*} 
By using \eqref{Ism-eq:2} and \eqref{Ism-eq:3}, one can obtain the relation
\begin{equation}\label{importante2}
 - (\nabla \theta_{Se} \times \nabla (u_{Te} + fx))/s \times u_{Se} 
= -\frac{g}{\theta_0} z \nabla \theta_{Se} -  fx \nabla u_{Te} + \nabla(u_{Te}^2/2 + u_{Te} fx).
\end{equation}
Now, put together $(\ref{importante1})$ and $(\ref{importante2})$ to derive 
\begin{equation} \label{unaecuacion}
- (q_e /s) \times u_{Se}  - \nabla \left(p_e + |u_{Se}|^2/2 + u_{Te}^2/2 -  \frac{g}{\theta_0} z \theta_{Se} \right)  = 0.
\end{equation}
By dotting \eqref{unaecuacion} against $u_{Se},$ one obtains
$$u_{Se} \cdot \nabla \left(p_e + |u_{Se}|^2/2 + u_{Te}^2/2  - \gamma_S \theta_{Se} \right)  = 0,$$
which, together with the equation for conservation of potential vorticity \eqref{generalisedvorti}, yields the following Bernoulli condition for the ISM:
$$p_e + |u_{Se}|^2/2 + u_{Te}^2/2  -\gamma_S \theta_{Se} = K(q_e).$$
Here $K$ stands for a real differentiable function. Note that this becomes the Bernoulli condition for incompressible 2D Euler upon substitution of $\theta_S=0,$ $u_T =0.$ To provide a explicit formula for the function $K$, we express \eqref{unaecuacion} as
$$q_e \hat{y} \times u_{Se} = -s \nabla K(q_e).$$
Applying cross product with $\hat{y}$ on the left-hand side we have that
$$-q_e u_{Se} = - s \hat{y} \times \nabla K(q_e).$$
Therefore, by taking into account the equilibrium condition \eqref{intro:conditions:equi:2},
$$ q_e \text{curl}(\Phi'(q_e) \hat{y}) = - \hat{y} \times \nabla K(q_e).$$
Rewrite this as
$$ q_e \nabla^T (\Phi'(q_e)) = \nabla^T (K(q_e)),$$
and applying the chain rule 
$$ q_e \Phi''(q_e) \nabla^T q_e =  K'(q_e) \nabla^T q_e,$$
so we infer that
$$ q_e \Phi''(q_e)  =  K'(q_e).$$
Hence, upon integration, we find that the function $\Phi$ can be expressed as
$$\Phi(\lambda) = \lambda \left( \int_\lambda \frac{K(t)}{t^2} \diff t + C  \right) ,$$
where $C$ is an integration constant. Also, since $\nabla K(q) = K'(q) \nabla q,$ we obtain
\begin{eqnarray}  \label{relacionfun}
\Phi''(q_e)= \frac{(\hat{y} \times \nabla q_e) \cdot u_{Se}}{ |\nabla q_e|^2}.
\end{eqnarray}

\end{proof}
\begin{remark}
Relation \eqref{relacionfun} is fundamental when deriving formal and nonlinear stability conditions for the ISM.
\end{remark}

\subsection{Formal stability conditions for the ISM}
To study the formal stability of the ISM around our restricted class of equilibrium solutions, we need to calculate the second variation of $H_\Phi,$ which reads
$$\delta^2 H_\Phi = \int_\Omega \left\{ |\delta u_S|^2 + (\delta u_T)^2 \right\} \diff V + \int_\Omega \Phi''(q) (\delta q)^2 \diff V.$$
Therefore, one obtains a straightforward (albeit quite general) condition for formal stability, namely
$$\Phi''(q_e) >0.$$

\begin{theorem}[Formal stability conditions for the ISM] \label{formalth}
An equilibrium point of the ISM belonging to the restricted class specified in Theorem \ref{th:equilibria} is formally stable if $\Phi''(q_e)>0.$ This is, an equilibrium point is formally stable if 
$$\frac{(\hat{y} \times \nabla q_e) \cdot u_{Se}}{ |\nabla q_e|^2} >0.$$
\end{theorem}

\subsection{Nonlinear stability for the ISM}
To derive nonlinear stability conditions for the ISM, we follow the Energy-Casimir algorithm \ref{stepa}-\ref{stepf}. We need to construct two quadratic forms $Q_1$ and $Q_2$ depending on the variables $\delta u_S,$ $\delta \theta_S,$ $\delta u_T.$ Note that $Q_1$ has to satisfy
\begin{align*}
Q_1(\Delta u_S, \Delta u_T, \Delta \theta_S) &\leq h(u_{Se} + \Delta u_S, u_{Te} + \Delta u_T, \theta_{Se} + \Delta \theta_S) - h(u_{Se}, u_{Te}, \theta_{Se}) \\
&- Dh(u_{Se}, u_{Te}, \theta_{Se}) \cdot (\Delta u_S, \Delta u_T, \Delta \theta_S).
\end{align*}
Since the conserved Hamiltonian $h$ is quadratic plus a linear term on $\theta_S$ (see \eqref{EPEady-erg1}), this can be done by choosing $Q_1 = h + \gamma_S \theta_S.$ Next, since $Q_2$ has to satisfy
\begin{align*}
Q_2(\Delta u_S, \Delta u_T, \Delta \theta_S) &\leq C_\Phi(u_{Se} + \Delta u_S, u_{Te} + \Delta u_T, \theta_{Se} + \Delta \theta_S) - C_\Phi (u_{Se}, u_{Te}, \theta_{Se}) \\
& - D C_\Phi (u_{Se}, u_{Te}, \theta_{Se}) \cdot (\Delta u_S, \Delta u_T, \Delta \theta_S),
\end{align*}
we select
$$Q_2 = \lambda_1 \int_\Omega (\Delta q)^2 \diff V,$$
where $\lambda_1 \in \mathbb{R}$ is such that 
\begin{center}
$\lambda_1 \leq \Phi''(x),$ \quad for all $x \in \mathbb{R}.$ 
\end{center}
Condition \ref{stepd} in the Energy-Casimir algorithm requires 
\[ Q_{1}+Q_{2}= h +  \lambda_1 \int_\Omega (\Delta q)^2 \diff V > 0, \]
which is guaranteed for instance if $\lambda_1 >0.$ We point out that $H_\Phi$ is continuous with respect to the defined norm
$$\norm{(\delta u_S, \delta u_T, \delta \theta_S)}_{Q_1 + Q_2}=\int_\Omega \left \{  \frac12 |\delta u_S|^2 + \frac12 (\delta u_T)^2  \right \} \diff V +  \lambda_1 \int_\Omega (\delta q)^2 \diff V,$$
provided $\Phi''(x) \leq \lambda_2 < \infty,$ for all $x \in \mathbb{R}.$ \\\\
Finally, we show the estimate required in \ref{stepe}:
\begin{align*}
Q_1 + Q_2 &= \int_\Omega \left\{ |\Delta u_S(t)|^2 + (\Delta u_T(t))^2 \right\} \diff V + \lambda_1 \int_\Omega  (\Delta q(t))^2 \diff V \\
 &\leq \int_\Omega \left\{ \frac{1}{2} (|u_{S}(0)|^2 + u_{T}(0)^2)  - \gamma_S \theta_{S}(0) \right\} \diff{V} + \int_\Omega \Phi(q(0)) \diff{V} \\
&- \int_\Omega \left\{ \frac{1}{2} (|u_{Se}|^2 + u_{Te}^2)  - \gamma_S \theta_{Se} \right\} \diff{V} 
- \int_\Omega \Phi(q_e) \diff{V} .
\end{align*}
Taking into account the construction above, we can derive the following theorem.
\begin{theorem}[Nonlinear stability conditions for the ISM]
We can define a norm $Q$ on $\mathfrak{X}(\Omega) \circled{s} \mathcal{F}(\Omega) \times \wedge^2 (\Omega)$ such that an equilibrium point of the ISM belonging to the restricted class specified in Theorem \ref{th:equilibria} is nonlinearly stable with respect to $Q$ if
$$0 < \lambda_1 \leq \frac{(\hat{y} \times \nabla q_e) \cdot u_{Se}}{ |\nabla q_e|^2} \leq \lambda_2 < \infty.$$
\end{theorem}

\section{Local well-posedness of the ISM}\label{3}
In this section, we establish the local existence and uniqueness of solutions of \eqref{Ism-eq:1}-\eqref{Ism-eq:4} on bounded domains $\Omega\subset\mathbb{R}^{2}$ with smooth boundary $\partial\Omega$ satisfying the boundary condition \eqref{Ism-eq:bdy}. We will assume that the constants $f=s=\theta_{0}=g=1$ without loss of generality. Therefore, the equations can be written as
\begin{subequations}
\begin{align}
 \label{ISM1} \partial_t u_{S} + u_{S}\cdot\nabla u_{S}
-u_{T}{\hat{x}}
= -\nabla p 
 + \theta_{S}\hat{z},  \\
\label{ISM2}\partial_t u_{T} + u_{S}\cdot\nabla u_{T}
+ u_{S}\cdot\hat{{x}}=-z, \\
 \label{ISM3} \partial_t \theta_{S} + u_{S}\cdot\nabla\theta_{S} + u_{T} = 0,\\
 \label{ISM4}  \nabla \cdot u_{S} =  0, 
\end{align}
\end{subequations}
with boundary condition
\begin{equation}\label{ISM5}
u_S \cdot n = 0 \text{ \ on  }  \partial \Omega.
\end{equation}
We prove the following theorem:
\begin{theorem}\label{mainth}
For $s > 2$ integer and initial data $(u^{0}_{S},u^{0}_{T},\theta^{0}_{S})\in H^{s}_\star(\Omega) \times H^{s} (\Omega) \times H^{s}(\Omega)$, there exists a time $T=T(\norm{(u^{0}_{S},u^{0}_{T},\theta^{0}_{S})}_{H^{s}})>0$ such that the equations \eqref{ISM1}-\eqref{ISM4} with boundary condition \eqref{ISM5} have a unique solution $(u_{S},u_{T},\theta_{S})$ in $C([0,T]; H^s_\star \times H^s \times H^s).$ 
\end{theorem}

Before starting with the proof of Theorem \ref{mainth}, we project our equations \eqref{ISM1}-\eqref{ISM4} by using the Leray's projector (see Lemma \ref{HelmHodge}) $\mathcal{P}$
into the following new system of equations
\begin{subequations}
\begin{align}
\label{ISM1P} \partial_t u_{S} + \mathcal{P} (u_{S}\cdot\nabla u_{S}) 
-\mathcal{P} (u_{T}\hat{x}) &= \mathcal{P} (\theta_{S}\hat{z}) ,  \\
 \label{ISM2P}\partial_t u_{T} + u_{S}\cdot\nabla u_{T}
+ u_{S}\cdot\hat{{x}}&= -z,  \\
\label{ISM3P} \partial_t \theta_{S} + u_{S}\cdot\nabla\theta_{S} + u_{T}  &= 0, \\
\label{ISM4P} \mathcal{P} u_{S} &= u_{S}.
\end{align}
\end{subequations}
We will show that we can go from \eqref{ISM1P}-\eqref{ISM4P} to \eqref{ISM1}-\eqref{ISM4} by solving a Poisson problem for the pressure. The following lemma is well-known \cite{Lion},

\begin{lemma}\label{Neumann}(Neumann problem) Given $f\in W^{k,p}(\Omega)$ for $k \in \mathbb{N}$, and $g\in W^{k+1-\frac{1}{p},p}(\partial\Omega)$ satisfying the compatibility condition
\[\int_{\Omega} f \diff V=\int_{\partial\Omega} g \diff V,\]
there exists $\phi\in W^{k+2,p}(\Omega)$ satisfying
\begin{align}
\begin{split}
\Delta \phi &=f, \text{ \ in \ } \Omega, \\
\frac{\partial\phi}{\partial n} &=g, \text{ \ on \ } \partial\Omega.
\end{split}
\end{align}
Moreover, 
\[ \norm{\nabla \phi}_{W^{k+1,p}(\Omega)}\lesssim \norm{f}_{W^{k,p}(\Omega)} + \norm{g}_{W^{k+1-\frac{1}{p},p}(\partial\Omega)} .\]
\end{lemma}\vspace{0.3cm}
Let $F(f,g)$ denote the operator $f \cdot \nabla g,$ where $f,g$ can be vector fields on $\Omega$ or a scalar functions. In order to prove that \eqref{ISM1P}-\eqref{ISM4P} admits a unique local strong solution, we will apply Theorem \ref{abstractevolution} to the following variation of the aforementioned equations, in which the solution is sought in $H^s \times H^s \times H^s$ rather than in $H^s_\star \times H^s \times H^s$:
\begin{subequations}
\begin{align}
   \label{ISMP5}\partial_t u_{S} + F(\mathcal{P} u_{S}, u_{S}) - \mathcal{Q} F (\mathcal{P} u_{S}, \mathcal{P} u_{S}) 
-\mathcal{P} (u_{T}{\hat{x}}) - \mathcal{P} (\theta_{S}\hat{z}) 
&= 0, \\
 \label{ISMP6} \partial_t u_{T} + F(\mathcal{P} u_{S},u_{T})
+ \mathcal{P} u_{S}\cdot\hat{{x}}
+z &= 0,  \\
 \label{ISMP7} \partial_t \theta_{S} + F(\mathcal{P} u_{S},\theta_{S}) + u_{T}  &= 0. 
\end{align}
\end{subequations}
We claim that the condition $\mathcal{P} u_{S} = u_{S}$ follows immediately for all the solutions in $C_w([0,T]; H^s)$ of equation \eqref{ISMP5}-\eqref{ISMP7}. Indeed, first note that
\begin{align}
(\mathcal{P} (u_{T} \hat{x} + \theta_{S} \hat{z}), \mathcal{Q} u_{S})_{L^2} &= 0, \label{cond1:Q}\\
(\mathcal{Q} F (\mathcal{P} u_{S}, \mathcal{P} u_{S}), \mathcal{Q} u_{S})_{L^2} &= (F (\mathcal{P} u_{S}, \mathcal{P} u_{S}), \mathcal{Q} u_{S})_{L^2}. \label{cond2:Q}
\end{align}
Therefore, using \eqref{cond1:Q} and \eqref{cond2:Q} we obtain
\begin{align*}
\frac{1}{2}\frac{d}{d t} \norm{\mathcal{Q} u_{S}}^{2}_{L^{2}} &= (\partial_{t} u_{S}, \mathcal{Q} u_{S})_{L^2} \\
&= - (F(\mathcal{P} u_{S}, u_{S} - \mathcal{P} u_{S}) , \mathcal{Q} u_{S} )_{L^2} = - (F(\mathcal{P} u_{S}, \mathcal{Q} u_{S}), \mathcal{Q} u_{S})_{L^2} = 0,
\end{align*}
since $(F(f,g),g)_{L^2} = 0$ if $f$ is divergence-free. Hence, any weak solution in $C_w([0,T]; H^s)$ of \eqref{ISMP5}-\eqref{ISMP7} satisfies $ \mathcal{Q} u_S = 0.$\\\\
Now we state some estimates (cf. \cite{KatLai},\cite{Kato2}) we will need in the proof of Theorem \ref{mainth}: \\
\begin{proposition} \label{estimadas} 
\begin{enumerate}[(i)] Set $s_{0}=3$ and $s\in\mathbb{N}$.
\item  Let $s \geq s_0$. 
\begin{enumerate} 
\item \label{est:basic:1a} Let $f \in H_{\star}^{s},$ $g \in H^{s+1}.$ We have that
\begin{align}
\norm{F(f,g)}_{H^s} &\lesssim \norm{f}_{H^{s_{0}-1}} \norm{g}_{H^{s+1}} + \norm{f}_{H^{s}} \norm{g}_{H^{s_0}}, \label{est:basic:1a'} \\
\abs{(g, F(f,g))_{H^{s}}}&\lesssim \norm{f}_{H^{s_0}} \norm{g}^{2}_{H^{s}} + \norm{f}_{H^{s}} \norm{g}_{H^{s}} \norm{g}_{H^{s_0}}.  \label{est:basic:1a''}
\end{align}
\item  \label{est:basic:1b} Let $f,g \in H_{\star}^{s}.$ Then $Q F(f,g) \in H^s$ and
$$\norm{QF(f,g)}_{H^{s}} \lesssim \norm{f}_{H^{s_0}} \norm{g}_{H^{s}} +  \norm{f}_{H^{s}} \norm{g}_{H^{s_0}}.$$ 
\item  \label{est:basic:1c} Let $f\in H_{\star}^{s},$ $g\in H_{\star}^{s+1}.$ We have
$$|(g, \mathcal{P} F(f,g))_{H^s}| \lesssim \norm{f}_{H^{s_{0}}} \norm{g}^{2}_{H^{s}}+ \norm{f}_{H^{s}} \norm{g}_{H^{s}} \norm{g}_{H^{s_0}}.$$
\end{enumerate}
\item  Let $1 \leq s \leq s_0-1.$ 
\begin{enumerate}
\item \label{est:basic:2a}For $f \in H_{\star}^{s_0},$ $g \in H^{s+1}$
\begin{align}
\norm{F(f,g)}_{H^{s}} &\lesssim \norm{f}_{H^{s_0-1}} \norm{g}_{H^{s+1}}, \label{est:basic:2a'}\\
\abs{(g, F(f,g))_{H^s}} &\lesssim \norm{f}_{H^{s_0}} \norm{g}^{2}_{H^{s}}. \label{est:basic:2a''}
\end{align}
\item \label{est:basic:2b}If $f\in H_{\star}^{s_0},$ $g\in H_{\star}^{s}$ we have
$$\norm{Q(f,g)}_{H^s} \lesssim \norm{f}_{H^{s_0}} \norm{g}_{H^{s}}.$$
\item \label{est:basic:2c} If $f\in H_{\star}^{s_0},$ $g\in H_{\star}^{s+1}$ then
$$\abs{(f, \mathcal{P} F(f,g))_{H^s}}\lesssim \norm{f}_{H^{s_0}} \norm{g}^{2}_{{H^s}}.$$
\end{enumerate}
\end{enumerate}
\end{proposition}
We are ready to start with the proof of Theorem \ref{mainth}.
\begin{proof}[Proof of Theorem \ref{mainth}]
We wish to apply Theorem \ref{abstractevolution} to equations \eqref{ISMP5}-\eqref{ISMP7}. To do this we need to construct an admissible triplet $\{V,H,X\}.$ Define $X= H^0 \times H^0 \times H^0,$ and for $s \geq s_0$ set $H = H^{s, s_{0}}$ to be the Hilbert space equipped with the norm
$$((u_{S},u_{T},\theta_{S}), (u'_{S},u'_{T},\theta'_{S}))_{H} $$
$$ = (u_{S},u'_{S})_{H^{s_0}} + (u_{S}, u'_{S})_{H^s} + (u_{T},u'_{T})_{H^{s_0}} + (u_{T}, u'_{T})_{H^s} + (\theta_{S},\theta'_{S})_{H^{s_0}} + (\theta_{S}, \theta'_{S})_{H^s}. $$
Note that since $s \geq s_0,$ the $H$-norm is equivalent to the $H^s$-norm because of the inequalities
$$\norm{f}^{2}_{H^{s}} \leq \norm{f}^{2}_{H} \leq 2 \norm{f}^{2}_{H^{s}}, \quad f \in H.$$
We are left to construct $V.$ This is selected as the subspace of $H^s$ corresponding to the domain of the self-adjoint unbounded nonnegative operator $S$ taking values in $H^0$ defined by 
$$S = \displaystyle\sum_{|\alpha| \leq s_0} (-1)^{\alpha}D^{2\alpha} + \displaystyle\sum_{|\alpha| \leq s} (-1)^{\alpha}D^{2\alpha},$$
with Neumann boundary conditions. Then one has (cf. \cite{Nire}, \cite{Lions})
$$V \subset H^{2s} \subset H^{s+1}.$$
Note that we have not used $H_{*}^{m}$ or $H_{*}^{0}$ to define the admissible triplet above. This is due to the extra complications that arise from the  divergence and boundary conditions in order to define a suitable space $V$. \\\\ We apply Theorem \ref{abstractevolution} to the operator
$$A(t,(u_{S},u_{T},\theta_{S})) = \left(A_1(t,(u_{S},u_{T},\theta_{S})), A_2(t,(u_{S},u_{T},\theta_{S})), A_3(t,(u_{S},u_{T},\theta_{S}))\right),$$
where
\begin{align*}
A_1 &= F(\mathcal{P} u_{S}, u_{S}) - \mathcal{Q} F (\mathcal{P} u_{S}, \mathcal{P} u_{S}) 
-\mathcal{P} (u_{T}{\hat{x}}) - \mathcal{P} (\theta_{S}\hat{z}), \\
 A_2 &= F(\mathcal{P} u_{S},u_{T}) + \mathcal{P} u_{S}\cdot\hat{{x}} - z, \\
A_3 &= F(\mathcal{P} u_{S},\theta_{S}) + u_{T}.
\end{align*}
Due to Proposition \ref{estimadas} (\ref{est:basic:1a}-\ref{est:basic:1b} and \ref{est:basic:2a}-\ref{est:basic:2b}), the operator $A$ maps $[0,T] \times H = [0,T] \times H^s \times H^s \times H^s$ into $H^{s-1} \times H^{s-1} \times H^{s-1}$ $\subset H^{0}=X$ weakly continuously. Now we bound
\begin{align*}
((u_{S},u_{T},\theta_{S}), A(t,(u_{S},u_{T},\theta_{S})))_H  
 &= (u_{S}, F(\mathcal{P} u_{S}, u_{S}) - \mathcal{Q} F (\mathcal{P} u_{S}, \mathcal{P} u_{S}) 
-\mathcal{P} (u_{T}{\hat{x}}) - \mathcal{P} (\theta_{S}\hat{z}))_{H^{s_0}} \\
&+ (u_{S}, F(\mathcal{P} u_{S}, u_{S}) - \mathcal{Q} F (\mathcal{P} u_{S}, \mathcal{P} u_{S}) -\mathcal{P} (u_{T}{\hat{x}}) - \mathcal{P} (\theta_{S}\hat{z}) )_{H^{s}} \\
&+ (u_{T}, F(\mathcal{P} u_{S} , u_{T}) + \mathcal{P} u_{S} \cdot \hat{{x}} - z)_{H^{s_0}}  \\ 
&+  (u_{T}, F(\mathcal{P} u_{S},u_{T}) +  \mathcal{P} u_{S} \cdot\hat{{x}} - z )_{H^{s}} \\
&+ (\theta_{S}, F(\mathcal{P} u_{S},\theta_{S}) + u_{T})_{H^{s_0}} + (\theta_{S}, F(\mathcal{P} u_{S},\theta_{S}) + u_{T})_{H^{s}}. 
\end{align*}
By using \ref{est:basic:1a} in Proposition \ref{estimadas}, we obtain the estimates
\begin{align*}
\abs{(u_{S}, F(\mathcal{P} u_{S}, u_{S})_{H^{s_0}}} &\lesssim ||u_{S}||^{3}_{H^{s_0}}, \\
\abs{ (u_{T}, F(\mathcal{P} u_{S}, u_{T}))_{H^{s_0}}} &\lesssim \norm{u_{S}}_{H^{s_0}} \norm{u_{T}}^{2}_{H^{s_0}}, \\
\abs{(\theta_{S}, F(\mathcal{P} u_{S}, \theta_{S}))_{H^{s_0}}} &\lesssim \norm{u_{S}}_{H^{s_0}} \norm{\theta_{S}}^{2}_{H^{s_0}}.
\end{align*}
Moreover, by \ref{est:basic:1a}
$$\abs{(u_{S}, \mathcal{Q} F (\mathcal{P} u_{S}, \mathcal{P} u_{S}))_{H^{s_0}}} \lesssim \norm{u_S}_{H^{s_0}} \norm{\mathcal{Q} F (\mathcal{P} u_{S}, \mathcal{P} u_{S})}_{H^{s_0}} \lesssim \norm{u_{S}}^{3}_{H^{s_0}}.$$
In the same way we can derive the bounds:
\begin{align*}
\abs{(u_{S}, F(\mathcal{P} u_{S}, u_{S}))_{H^{s}} } &\lesssim \norm{u_{S}}_{H^{s_0}} \norm{u_{S}}^{2}_{H^{s}}, \\
\abs{(u_{T}, F(\mathcal{P} u_{S}, u_{T}))_{H^{s}} } &\lesssim  \norm{u_{S}}_{H^{s_0}} \norm{u_{T}}^{2}_{H^{s}} + \norm{u_{S}}_{H^{s}} \norm{u_{T}}_{H^{s}}\norm{u_{T}}_{H^{s_{0}}} , \\
\abs{(\theta_{S}, F(\mathcal{P} u_{S}, \theta_{S}))_{H^{s}} } &\lesssim  \norm{u_{S}}_{H^{s_0}} \norm{\theta_{S}}^{2}_{H^{s}} + \norm{u_{S}}_{H^{s}} \norm{\theta_{S}}_{H^{s}} \norm{\theta_{S}}_{H^{s_{0}}}.
\end{align*}
Also, by \ref{est:basic:1a}
$$|(u_S, \mathcal{Q} F (\mathcal{P} u_S, \mathcal{P} u_S))_{H^s}| \lesssim ||u_S||_{H^s} || \mathcal{Q} F (\mathcal{P} u_S, \mathcal{P} u_S) ||_{H^s} \lesssim ||u_S||_{H^s}^2 ||u_S||_{H^{s_0}}.$$
Note that the linear terms $\mathcal{P} (u_{T}{\hat{x}}), \mathcal{P} (\theta_{S}\hat{z}), u_{S} \cdot \hat{x}, z, u_{T},$ can be straightforwardly estimated. Therefore
\begin{align}
((u_{S},u_{T},\theta_{S}), A(t,(u_{S},u_{T},\theta_{S})))_H  &\lesssim \left(\norm{u_{S}}_{H^{s_0}}  + \norm{u_T}_{H^{s_0}}  + \norm{\theta_{S}}_{H^{s_0}}\right)^3 \nonumber \\
 &+ \left(\norm{u_{S}}_{H^{s_0}}  + \norm{u_{T}}_{H^{s_0}}  + \norm{\theta_{S}}_{H^{s_0}} + 1\right)  \left(\norm{u_{S}}_{H^{s}}  + \norm{u_{T}}_{H^{s}}  + \norm{\theta_{S}}_{H^{s}}\right)^2 \nonumber \\
&\lesssim  \norm{(u_{S},u_{T},\theta_{S})}_H^3 +  \norm{(u_{S},u_{T},\theta_{S})}_H^2 + \norm{(u_{S},u_{T},\theta_{S})}_H \nonumber \\
& \label{last:estimate} \lesssim \left( \norm{(u_{S},u_{T},\theta_{S})}_{H} +1 \right)^3.
\end{align}
Hence, Theorem \ref{abstractevolution} can be applied to equations \eqref{ISMP5}-\eqref{ISMP7} with initial data $(u^{0}_{S}, u^{0}_{T}, \theta^{0}_{S})$ in $H^s,$ guaranteeing the existence of a locally weak solution $(u_{S},u_{T},\theta_{S}) \in C_w([0,T]; H^s)$ for any integer $s>2$. \\\\
Estimate \eqref{last:estimate} is not sharp. Obtaining better estimates for the existence time is just a matter of carrying out the estimates more carefully. However, for the sake of exposition clarity it was convenient to write them as we did. The bound condition in \eqref{condicion} is satisfied with
$$\beta(r) = \dfrac{K}{2}(r+1)^{3/2},$$
so the differential equation \eqref{weak} (to be solved in order to obtain a local existence time) becomes
\begin{equation}\label{differential:equation}
\begin{cases}
\rho_t = K (\rho + 1)^{3/2}, \\
\rho(0)= \norm{\phi}^{2}_{H}.
\end{cases}
\end{equation}

The solution of this ODE can be explicitly written as
\begin{equation*}
 \rho(t) = \frac{-2 C K t - C^2 - (2K t)^2 + 4}{(C + K t)^2}, \quad \text{with  \quad } C = \frac{2}{\sqrt{\norm{\phi}^{2}_{H}+1}}.
\end{equation*}
This means the solution of \eqref{differential:equation} is at least guaranteed to exist until 
\[T =  \dfrac{2}{K\sqrt{\norm{\phi}^{2}_{H}+1}}.\]

To complete the proof of Theorem \ref{mainth}, we need to justify the regularity and uniqueness of the solutions. First, we focus on showing uniqueness of the weak solution. Indeed, first note that $(u_{S},u_{T},\theta_{S}) \in \text{Lip} ([0,T]; H^0_{\star})$ since $\partial_t (u_{S},u_{T},\theta_{S}) \in C_w([0,T]; H^0_{\star}),$ so $\norm{(u_{S},u_{T},\theta_{S})}_{L^{2}}$ is differentiable in time. To show uniqueness we argue by contradiction. Suppose that $(u^{1}_{S},u^{1}_{T},\theta^{1}_{S})$ and $(u^{2}_{S},u^{2}_{T},\theta^{2}_{S})$ are two different solutions of \eqref{ISMP5}-\eqref{ISMP7} with the same initial data $(u^{0}_{S},u^{0}_{T},\theta^{0}_{S})$. Define the differences $\tilde{u_{S}} = u^{1}_{S}-u^{2}_{S},$ $\tilde{u_{T}} = u^{1}_{T} - u^{2}_T,$ $\tilde{\theta_{S}} = \theta^{1}_{S} - \theta^{2}_{S}.$ We have
\begin{eqnarray*}
\frac{d}{d t} \frac{1}{2} \norm{\tilde{u_{S}}}^{2}_{L^{2}} &=& -( \mathcal{P} (u^{1}_{S} \cdot \nabla \tilde{u_{S}}) , \tilde{u_{S}} )_{L^{2}} - ( \mathcal{P} (\tilde{u_{S}} \cdot \nabla u^{2}_{S}) , \tilde{u_{S}} )_{L^2} + ( \mathcal{P} (\tilde{u_{T}} \hat{x}) , \tilde{u_{S}} )_{L^{2}} + (\mathcal{P} (\tilde{\theta_{S}} \hat{z}) , \tilde{u_{S}} )_{L^{2}} , \\
\frac{d}{dt} \frac{1}{2} \norm{\tilde{u_{T}}}^{2}_{L^{2}} &=& -(\mathcal{P} u^{1}_{S} \cdot \nabla \tilde{u_{T}}, \tilde{u_{T}} )_{L^2} - ( \mathcal{P}\tilde{u_{S}} \cdot \nabla u^{2}_{T}, \tilde{u_{T}} )_{L^{2}} - (\mathcal{P} \tilde{u_{S}} \cdot \hat{x}, \tilde{u_{T}} )_{L^2}, \\
\frac{d}{dt} \frac{1}{2} \norm{\tilde{\theta_{S}}}^{2}_{L^{2}}  &=& -( \mathcal{P}u^{1}_{S} \cdot \nabla \tilde{\theta_{S}}, \tilde{\theta_{S}} )_{L^{2}} - ( \mathcal{P} \tilde{u_{S}} \cdot \nabla \theta^{2}_{S}, \tilde{\theta_{S}} )_{L^2} - ( \tilde{u_{T}}, \tilde{\theta_{S}} )_{L^{2}}.
\end{eqnarray*} 
By using the properties of $\mathcal{P}$ stated in Lemma \ref{HelmHodge} we observe that
$$( \mathcal{P} (u^{1}_{S} \cdot \nabla \tilde{u_{S}}), \tilde{u_{S}})_{L^{2}} = ( u^{1}_{S} \cdot \nabla \tilde{u_{S}}, \mathcal{P}\tilde{u_{S}})_{L^2} =  (u^{1}_{S} \cdot \nabla \tilde{u_{S}}, \tilde{u_{S}})_{L^2} = 0, $$
and
$$(\mathcal{P} u^{1}_{S} \cdot \nabla \tilde{u_{T}}, \tilde{u_{T}})_{L^2} = 0, \quad ( \mathcal{P} u^{1}_{S} \cdot \nabla \tilde{\theta_{S}}, \tilde{\theta_{S}} )_{L^2} = 0.$$
Hence obtaining
\[ \frac{d}{dt}  \left( \norm{\tilde{u_{S}}}^2_{L^2} + \norm{\tilde{u_{T}}}^2_{L^2}  + \norm{\tilde{\theta_{S}}}^2_{L^2}  \right) \lesssim \left( \norm{\nabla u^{2}_{S}}_{L^\infty} + \norm{\nabla u^{2}_{T}}_{L^\infty} + \norm{\nabla \theta^{2}_{S}}_{L^\infty} + 1\right)\left(  \norm{\tilde{u_{S}}}^2_{L^2} + \norm{\tilde{u_{T}}}^2_{L^2}  + \norm{\tilde{\theta_{S}}}^2_{L^2} \right), \]
from which, after a standard application of Gr\"onwall's argument, uniqueness follows.\\\\
We have shown we can construct a unique weak local solution of \eqref{ISMP5}-\eqref{ISMP7}. Moreover, the solution actually enjoys higher regularity. This is the last point we need to discuss to conclude our argument. We demonstrate this in four steps: \newline \\
\textit{Step 1: } First note that although the solution $(u_{S}(t),u_{T}(t),\theta_{S}(t))$ of \eqref{ISMP5}-\eqref{ISMP7} was sought in $C_w([0,T];H^s\times H^{s}\times H^{s})$ with initial data $\phi,$ we actually have $\norm{u(t)}_{H} 
\rightarrow \norm{\phi}_{H}$ (strongly!). This holds because $\rho(t) \rightarrow ||\phi||_H^2$ for the solution of the differential equation \eqref{weak}.   \\ \\
\textit{Step 2: } By Step 1, $(u_{S},u_{T},\theta_{S})$ is right-continuous at $T=0.$ It is easy to check that it must also be right-continuous on $[0,T]$ by uniqueness (for $t \in [0,T]$ consider the same PDE with initial data $(u_{S},u_{T},\theta_{S})(t, \cdot)$). \\ \\
\textit{Step 3: } We need to prove that $(u_{S},u_{T},\theta_{S})$ is also left-continuous. However, since the ISM system \eqref{ISM1}-\eqref{ISM4} is not time reversible, the customary argument does not apply. We make the following change of variables: 
\[ \tilde{u}_{S}(X,t):=u_{S}(-X,-t), \quad \tilde{u}_{T}(X,t):=-u_{T}(-X,-t), \quad \tilde{\theta}_{S}(X,t):=-\theta_{S}(-X,-t), \quad \text{with} \ X=(x,z). \]
The new functions satisfy the following system of equations
\begin{subequations}
\begin{align} 
 \label{ISMPR1}
  \partial_t \tilde{u}_{S} + \mathcal{P} (\tilde{u}_{S}\cdot\nabla \tilde{u}_{S}) 
-\mathcal{P} (\tilde{u}_{T}{\hat{x}})
&= 
 \mathcal{P} (\tilde{\theta}_{S}\hat{z}),  \\
 \label{ISMPR2}\partial_t \tilde{u}_{T} + \tilde{u}_{S}\cdot\nabla \tilde{u}_{T}
+ \tilde{u}_{S}\cdot\hat{{x}}
&=   -z ,  \\
 \label{ISMPR3}\partial_t \tilde{\theta}_{S} + \tilde{u}_{S}\cdot\nabla\tilde{\theta}_{S} - \tilde{u}_{T}  &= 0, \\
 \label{ISMPR4} \mathcal{P} \tilde{u}_{S} &= \tilde{u}_{S}.
\end{align}
\end{subequations}
The operator $\tilde{A}$ associated with \eqref{ISMPR1}-\eqref{ISMPR4}  satisfies estimates which are similar to the ones satisfied by $A$ \eqref{ISMP5}-\eqref{ISMP7}. Therefore we can apply the same weak existence and uniqueness theorem in $C_w([0,\tilde{T}];H^s\times H^{s}\times H^{s})$ with initial data in $H^s$ for our new variables $(\tilde{u}_{S}(X,t),\tilde{u}_{T}(X,t),\tilde{\theta}_{S}(X,t))$. The existence time $\tilde{T}$ of the new tilde solutions needs not be the same as the one for the standard ones. Also, by Step 2, the new tilde solutions of \eqref{ISMPR1}-\eqref{ISMPR4} are necessarily right-continuous. \\

\textit{Step 4: } Now let $\left(u_{S},u_{T},\theta_{S}\right) \in C_w([0,T]; H^{s}\times H^{s}\times H^{s})$ solve \eqref{ISMP5}-\eqref{ISMP7}. We need to prove that $(u_{S},u_{T},\theta_{S})$ is left-continuous. For this, construct \[ \left(\tilde{u}_{S}(t),\tilde{u}_{T}(t),\tilde{\theta}_{S}(t) \right) = \left(u_{S}(T-t, \cdot),{u}_{T}(T-t,\cdot),\theta_{S}(T-t,\cdot)\right) \in C_w([0,T]; H^s \times H^{s} \times H^{s}),\]  with $\tilde{\phi} = \left( \tilde{u_{S}}(0, \cdot),\tilde{u}_{T}(0, \cdot),\tilde{\theta}_{S}(0,\cdot)\right) \in H^{s}\times H^{s}\times H^{s}.$ Therefore $\left(\tilde{u}_{S},\tilde{u}_{T},\tilde{\theta}_{S}\right)$ has to be the unique weak local solution of \eqref{ISMPR1}-\eqref{ISMPR4} with initial data $\tilde{\phi}.$ By Step 3, $\left(\tilde{u}_{S},\tilde{u}_{T},\tilde{\theta}_{S}\right)$ is right-continuous, and hence $(u_{S},u_{T},\theta_{S})$ is left-continuous.\ \\
We have shown that there exists a local solution of \eqref{ISMP5}-\eqref{ISMP7} in $C([0,T]; H^{s}\times H^{s}\times H^{s})$, concluding the proof.

\end{proof}

For the sake of completeness we also include, without proof, the following result which states the continuous dependence of the solution on the initial data:
\begin{theorem}  \label{smooth}
Let $s>2$ be an integer. Let $(u_{S},u_{T},\theta_{S}) \in C([0,T]; H^s_\star \times H^s \times H^s)$ be the solution of \eqref{ISM1}-\eqref{ISM4} with initial data $(u^0_{S},u^0_{T},\theta^0_{S}) \in H^s_\star \times H^s \times H^s.$ Let $(u^{j}_{S}, u^{j}_{T}, \theta^{j}_{S})$ be the solution of \eqref{ISM1}-\eqref{ISM4} with initial data $(u^{0,j}_{S},u^{0,j}_{T},\theta^{0,j}_{S}) \in H^s_\star \times H^s \times H^s$$j=1,2,\ldots$, such that $(u^{0,j}_{S},u^{0,j}_{T},\theta^{0,j}_{S}) \rightarrow (u^0_{S},u^0_{T},\theta^0_{S})$ in $H^s \times H^s \times H^s.$ Then  
$$(u^j_{S},u^j_{T},\theta^j_{S}) \rightarrow (u_{S},u_{T},\theta_{S}), \quad in \quad C([0,T]; H^s \times H^s \times H^s).$$
\end{theorem}

\begin{remark}
We do not provide the proof of Theorem \ref{smooth} since it can be performed by imitating the arguments in \cite{KatLai}.
\end{remark}

\section{Blow-up criterion for the ISM}   \label{4}
In this section we prove the following blow-up criterion for the Incompressible Slice Model.
\begin{theorem}\label{blowupcriteria}
Assume that $(u^{0}_{S},u^{0}_{T},\theta^{0}_{S})\in H^{s}_\star (\Omega)\times H^{s} (\Omega)\times H^{s} (\Omega)$ for $s>2$ integer, and that the solution $(u_{S},u_{T},\theta_{S})$ of \eqref{ISM1}-\eqref{ISM4} with boundary condition \eqref{ISM5} is of class $C([0,T];  H_{\star}^{s}\times H^{s}\times H^{s})$. Then for $T^*< \infty,$ the following two statements are equivalent:
\begin{eqnarray}
&(i)& \quad E(t) < \infty, \quad \forall t < T^*, \quad \mbox{ and } \quad \limsup_{t \to T^*} E(t) = \infty,
\label{eq:X(t)toinfty} \\
&(ii)& \quad \int_0^{t} \norm{\nabla u_{S}(s)}_{L^{\infty}} \, \diff s < \infty, \quad \forall t < T^*, \quad \mbox{ and } \quad \int_0^{T^*}  \norm{\nabla u_{S}(s)}_{L^{\infty}} \, \diff s = \infty,
\label{eq:utoinfty}
\end{eqnarray}
where $E(t)=\norm{u_{S}}^2_{{H}^s}+\norm{u_{T}}^2_{{H}^s}+\norm{\theta_{S}}^2_{{H}^s}$. If such $T^*$ exists then it is called the first-time blow-up and \eqref{eq:utoinfty} is a blow-up criterion. 
\end{theorem}
\begin{remark}
Theorem \ref{blowupcriteria} can be a very useful tool for studying the problem of blow-up versus global existence in the Incompressible Slice Model. Note that, for instance, it can be applied to check whether a numerical simulation shows blow-up in finite time. 
\end{remark}
\begin{remark}
Notice that we cannot expect (as one has in 3D Euler), a Beale-Kato-Majda criterion, stating that if
\[\int_{0}^{T^*} \norm{\omega}_{\infty}< \infty,\]
then the corresponding solution stays regular on $[0,T^*]$. Here $\omega=\text{curl} \ u$. The main problem is that we cannot control properly $\theta_{S}$ and $u_{T}$ in terms of the vorticity only. If the equations were not coupled, this could be easily done by using a logarithmic inequality as in 3D Euler. However, in the ISM case it seems that controlling $\norm{\omega}_{\infty}$ is not enough for global regularity. 
\end{remark}
\begin{proof}[Proof of Theorem \ref{blowupcriteria}]
The proof is divided into several steps. \\ \\
\textit{Step 1: Estimates for $\norm{u_{S}}_{H^{s}}$, $\norm{u_{T}}_{H^{s}},$ and $\norm{\theta_S}_{H^{s}}$.} Setting $|\alpha|\leq s$, we apply $D^{\alpha}$ on both sides of \eqref{ISM1}, multiply by $D^{\alpha}u_{S}$ and integrate over $\Omega$ to obtain
\begin{align} 
\label{estimationu1} (D^{\alpha}\partial_{t}u_{S},D^{\alpha}u_{S})_{L^2}+(D^{\alpha}(u_{S}\cdot\nabla u_{S}),D^{\alpha} u_{S})_{L^2} &= (D^{\alpha}(u_{T}\hat{x}),D^{\alpha}u_{S})_{L^2}+(D^{\alpha}(\theta_{S}\hat{z}),D^{\alpha}u_{S})_{L^2}  \\
\label{estimationu2} &-(D^{\alpha}\nabla p, D^{\alpha}u_{S})_{L^2}. 
\end{align}
The second term on the left-hand side of \eqref{estimationu1} can be rewritten as
\begin{equation} \label{commutator}
 (D^{\alpha}(u_{S}\cdot \nabla u_{S}),D^{\alpha}u_{S})_{L^2}= (D^{\alpha}(u_{S}\cdot\nabla u_{S})-u_{S}\cdot D^{\alpha}\nabla u_{S}, D^{\alpha}u_{S})_{L^2}+(u_{S}\cdot D^{\alpha}\nabla u_{S},D^{\alpha}u_{S})_{L^2} .
\end{equation}
Therefore, applying the Cauchy-Schwarz inequality in \eqref{estimationu1}-\eqref{estimationu2}, we have that
\begin{align*}
 \frac{d}{dt} \norm{D^{\alpha} u_{S}}^{2}_{L^{2}} &\lesssim \norm{D^{\alpha}u_{S}}_{L^{2}}\norm{D^{\alpha}(u_{S}\cdot\nabla u_{S})-u_{S}\cdot D^{\alpha}\nabla u_{S}}_{L^{2}} \\
 &+ \norm{D^{\alpha}u_{S}}_{L^{2}}\left(\norm{D^{\alpha}(u_{T}\hat{x})}_{L^{2}}+\norm{D^{\alpha}(\theta_{S}\hat{z})}_{L^{2}}+\norm{D^{\alpha}\nabla p}_{L^{2}}\right), 
 \end{align*}
where we have used that the last term in \eqref{commutator} vanishes after integration by parts due to $\text{div} \ u_{S}=0$ and $u_{S}\cdot n=0$ on $\partial\Omega$.  Summing over $|\alpha| \leq s$, applying Young's inequality, and using \ref{calculusineq:2} in Lemma \ref{calculusineq} (with $f=u_{S}$ and $g=\nabla u_{S}$) yields
\begin{equation}\label{finalu}
\frac{d}{d t} \norm{u_{S}}^{2}_{H^{s}} \lesssim \norm{u_{S}}^{2}_{H^{s}}\left(\norm{u_{S}}_{W^{1,\infty}}+ 1 \right) + \norm{u_{T}}^{2}_{H^{s}}+\norm{\theta_{S}}^{2}_{H^{s}}+\norm{u_{S}}_{H^{s}}\norm{\nabla p}_{H^{s}}.
\end{equation}
To estimate $\norm{u_{T}}_{H^{s}}$ and $\norm{\theta_{S}}_{H^{s}}$ we proceed in a similar fashion. By applying these same techniques to \eqref{ISM2}, using  Lemma \ref{calculusineq}, and Cauchy-Schwarz inequality, we derive
\[ \frac{d}{d t} \norm{D^{\alpha}u_{T}}^{2}_{L^{2}} \lesssim \norm{D^{\alpha}u_T}_{L^{2}}\left(
\norm{u_{S}}_{W^{1,\infty}}\norm{u_{T}}_{H^{s}}+\norm{u_{S}}_{H^{s}}\norm{\nabla u_{T}}_{L^{\infty}} + \norm{D^{\alpha} u_{S}}_{L^{2}}+\norm{D^{\alpha}z}_{L^{2}} \right).
\]
Summing over $|\alpha|\leq s$ and using Young's inequality gives
\begin{equation}\label{finalv}
\frac{d}{d t} \norm{u_{T}}^{2}_{H^{s}} \lesssim \norm{u_{T}}^{2}_{H^{s}}\norm{u_{S}}_{W^{1,\infty}}+\norm{\nabla u_{T}}_{L^{\infty}}\norm{u_{S}}_{H^s}\norm{u_{T}}_{H^{s}}+\norm{u_{S}}^{2}_{H^{s}}+\norm{u_{T}}^{2}_{H^{s}}+1.
\end{equation}
A similar argument can be carried out to estimate $\norm{\theta_{S}}_{H^{s}}$, obtaining
\begin{equation}\label{finaltheta}
\frac{d}{d t} \norm{\theta_{S}}^{2}_{H^{s}} \lesssim \norm{\theta_{S}}^{2}_{H^{s}}\norm{u_{S}}_{W^{1,\infty}}+\norm{\nabla \theta_{S}}_{L^{\infty}}\norm{u_{S}}_{H^{s}}\norm{\theta_{S}}_{H^s}+\norm{\theta_{S}}^{2}_{H^{s}}+\norm{u_{T}}^{2}_{H^{s}}.
\end{equation}
Finally, by \eqref{finalu},\eqref{finalv}, and \eqref{finaltheta}, we conclude
\begin{equation}\label{BKM1}
\frac{d}{d t} E(t) \lesssim \left(\norm{u_{S}}_{W^{1,\infty}}+\norm{\nabla u_{T}}_{L^{\infty}}+\norm{\nabla \theta_{S}}_{L^{\infty}}+1\right)\left(E(t)+1\right) +\norm{\nabla p}_{H^{s}}\norm{u_{S}}_{H^{s}},
 \end{equation}
where $E(t)=\norm{u_{S}}^2_{{H}^s}+\norm{u_{T}}^2_{{H}^s}+\norm{\theta_{S}}^2_{{H}^s}$. \\\\
\textit{Step 2: Estimate the pressure term $\norm{\nabla p}_{H^{s}}$.} 
In order to do this, we need the following lemma.
\begin{lemma}\label{estimationpressure} If $u_{S},u_{T},\theta_{S}$, and $p$ satisfy equations \eqref{ISM1}-\eqref{ISM4} with boundary condition \eqref{ISM5}, then for $s\geq 3$ we have the following estimate
\[ \norm{\nabla p}_{H^{s}}  \lesssim \norm{u_{S}}_{H^{s}}\norm{u_{S}}_{W^{1,\infty}}+\norm{u_{T}}_{H^{s}
}+\norm{\theta_{S}}_{H^{s}}.\]

\end{lemma}

\begin{proof}[Proof of Lemma \ref{estimationpressure}]
The proof follows closely the lines of \cite{Temam} for the incompressible Euler equation. There, $\norm{\nabla p}_{H^{s}}$ is bounded in terms of $\norm{u}^{2}_{H^{s}}$. We perform some simple modifications of this idea in order to obtain convenient estimates. First take the divergence in \eqref{ISM1} and dot it against the outward normal $n=(n_{1},n_{2})$. Using the divergence-free condition and  $u_S \cdot n=0$ on $\partial \Omega$, we obtain the following  Neumann problem for $p$
\begin{align}
 \label{5.9}  \Delta p &= \text{div} \ (u_{T}\hat{x}+\theta_{S}\hat{z})- \displaystyle\sum_{i,j=1}^{2}D_{j}u_{S,i}\cdot D_{i}u_{S,j}, \ \text{in \ } \Omega, \\
 \label{5.10}\frac{\partial p}{\partial n}& = (u_{T}\hat{x}+\theta_{S}\hat{z})\cdot n-\displaystyle\sum_{i,j=1}^{2}u_{S,i}(D_{i}u_{S,j})n_{j}, \  \ \text{on \ } \partial\Omega. 
\end{align}
Proceeding as in \cite{Temam}, we can eliminate the derivatives from $u_{S}$ on the right hand side of \eqref{5.10} by representing $\partial\Omega$ locally as a level set of a smooth function, i.e. as $\phi(x)=0$, so that on every local patch we can write
\[ \frac{\partial p}{\partial n}=(u_{T}\hat{x}+\theta_{S}\hat{z})\cdot n-\displaystyle\sum_{i,j=1}^{2}u_{S,i}u_{S,j}\psi_{ij}, \quad \text{on \ } \partial\Omega, \]
where 
$$\psi_{ij}=\frac{D_{ij}\phi(x)}{|\nabla \phi(x)|}.$$ 
Notice that this sort of representation is only possible because the boundary $\partial\Omega$ is smooth enough. Hence we can estimate the pressure term $\norm{\nabla p}_{H^s}$ by applying \ref{calculusineq:1} in Lemma \ref{calculusineq}, combined with the Trace Theorem \ref{traceth}. Indeed, by \eqref{5.9}, \eqref{5.10}, and Lemma \ref{Neumann}:
\begin{align*}
 \norm{\nabla p}_{H^{s}} &\lesssim   \left( \left | \left| \text{div} (u_{T}\hat{x}+\theta_{S}\hat{z})- \displaystyle\sum_{i,j=1}^{2}D_{j}u_{S,i}\cdot D_{i}u_{S,j} \right| \right|_{H^{s-1}(\Omega)}+ \left| \left|(u_{T}\hat{x}+\theta_{S}\hat{z})\cdot n - \displaystyle\sum_{i,j=1}^{2}u_{S,i}u_{S,j}\psi_{ij} \right| \right|_{H^{s-\frac{1}{2}}(\partial \Omega)} \right) \\
 &\lesssim \norm{u_{T}\hat{x}+\theta_{S}\hat{z}}_{H^{s}(\Omega)}+ \left| \left|\displaystyle\sum_{i,j=1}^{2}D_{j}u_{S,i}\cdot D_{i}u_{S,j} \right| \right|_{H^{s-1}(\Omega)}+ \left| \left| \displaystyle\sum_{i,j=1}^{2}u_{S,i}u_{S,j}\psi_{ij} \right| \right|_{H^{s-\frac{1}{2}}(\partial \Omega)} \\
  &\lesssim  \norm{u_{T}\hat{x}+\theta_{S}\hat{z}}_{H^{s}(\Omega)}+\displaystyle\sum_{i,j=1}^{2}\left(\norm{D_{j}u_{S,i}}_{H^{s-1}(\Omega)}\norm{D_{i}u_{S,j}}_{L^{\infty}(\Omega)}+\norm{D_{j}u_{S,i}}_{L^\infty(\Omega)}\norm{D_{i}u_{S,j}}_{H^{s-1}(\Omega)}\right) \\ \ \ \ \ 
  &+ \displaystyle\sum_{i,j=1}^{2}\norm{u_{S,i}}_{H^{s}(\Omega)}\norm{u_{S,j}}_{L^\infty(\Omega)}+\norm{u_{S,j}}_{H^{s}(\Omega)}\norm{u_{S,i}}_{L^{\infty}(\Omega)}\\
 &\lesssim \norm{u_{S}}_{H^{s}(\Omega)}\norm{u_{S}}_{W^{1,\infty}(\Omega)}+\norm{u_{T}}_{H^{s}(\Omega)}+\norm{\theta_{S}}_{H^{s}(\Omega)}.
\end{align*}

\end{proof}
\textit{Step 3: Controlling $\norm{\nabla u_{T}}_{L^{\infty}}$ and  $\norm{\nabla \theta_{S}}_{L^{\infty}}$ by $\norm{\nabla u_{S}}_{L^{\infty}}.$ }
Take $\nabla$ in \eqref{ISM2} to obtain
\[ \partial_{t}\nabla u_{T}+\nabla(u_{S} \cdot \nabla u_{T})+\nabla(u_{S}\cdot \hat{x})=-\nabla z.\]
Let $p > 2$ be an integer and compute the $L^{2}$ inner product against $\nabla u_{T}\abs{\nabla u_{T}}^{p-2}$ in the last equation, deriving
\begin{equation*}
\begin{aligned}
(\partial_{t}\nabla u_{T}, \nabla u_{T} \abs{\nabla u_{T}}^{p-2})_{L^2}&+(\nabla (u_{S}\cdot \nabla u_{T}), \nabla u_{T} \abs{\nabla u_T}^{p-2})_{L^2} \\
&+(\nabla (u_{S} \cdot \hat{x}),\nabla u_{T} \abs{\nabla u_{T}}^{p-2})_{L^2}=(-\nabla z, \nabla u_{T} \abs{\nabla u_{T}}^{p-2})_{L^2}
\end{aligned}
\end{equation*}
The first term on the left-hand side is
\[ (\partial_{t}\nabla u_{T}, \nabla u_{T } \abs{\nabla u_{T}}^{p-2})_{L^2} = \frac{1}{p}\frac{d}{d t} \norm{\nabla u_{T}}^{p}_{L^{p}}.\]
We rewrite the second term as
\begin{equation}\label{rewrite}
(\nabla (u_{S}\cdot \nabla u_{T}), \nabla u_{T} \abs{\nabla u_{T}}^{p-2})_{L^2}=((\nabla u_{S}\cdot\nabla)u_{T},\nabla u_{T} \abs{\nabla u_{T}}^{p-2})_{L^2}+\frac{1}{p}(u_S, \nabla (\abs{\nabla u_{T}}^{p}))_{L^2}.
\end{equation}
Therefore, we obtain
\begin{align*}
 \frac{1}{p}\frac{d}{d t} \norm{\nabla u_{T}}^{p}_{L^{p}}= &-((\nabla u_{S}\cdot\nabla u_{T}),\nabla u_{T}\abs{\nabla u_{T}}^{p-2})_{L^2}-(\nabla (u_{S} \cdot\hat{x}),\nabla u_{T} \abs{\nabla u_{T}}^{p-2})_{L^2} \\
&-(\nabla z, \nabla u_{T} \abs{\nabla u_{T}}^{p-2})_{L^2}, 
\end{align*}
where we have taken into account that the second term in the right-hand side of \eqref{rewrite} vanishes after integration by parts. Now using H\"older and Young's inequality we derive
\begin{align}\label{finalLPV}
\frac{1}{p}\frac{d}{d t} \norm{\nabla u_{T}}^{p}_{L^{p}} &\lesssim  \norm{\nabla u_{S}}_{L^{\infty}}\norm{\nabla u_{T}}^{p}_{L^{p}}+\norm{\nabla u_{S}}_{L^{\infty}}\norm{\nabla u_{T}}^{p-1}_{L^{p}}+\norm{\nabla u_{T}}^{p-1}_{L^{p}} \nonumber \\
&\lesssim  \left(\norm{\nabla u_{S}}_{L^{\infty}}+1\right)\left( \norm{\nabla u_{T}}^{p}_{L^{p}}+1\right).
\end{align}
Proceeding as before, we can obtain a similar estimate for $\norm{\nabla \theta_{S}}_{L^{p}}$ in the form 
\begin{equation}\label{Lptheta}
\frac{1}{p}\frac{d}{d t} \norm{\nabla \theta_{S}}^{p}_{L^{p}} \lesssim \norm{\nabla u_{S}}_{L^{\infty}}\norm{ \nabla \theta_{S}}^{p}_{L^{p}}+\frac{1}{p}\norm{\nabla \theta_{S}}^{p-1}_{L^{p}}+ \frac{1}{p}\norm{\nabla u_{T}}^{p}_{L^{p}}.
\end{equation}
Putting together \eqref{finalLPV} and \eqref{Lptheta} we conclude
\[\frac{d}{d t} (\norm{ \nabla \theta_{S}}^{p}_{L^{p}}+\norm{\nabla u_{T}}^{p}_{L^{p}}) \lesssim p \left(\norm{\nabla u_{S}}_{L^{\infty}}+ 1\right)\left(\norm{\nabla \theta_{S}}^{p}_{L^{p}}+\norm{\nabla u_{T}}^{p}_{L^{p}}+1\right).\] 
It is important to note that the constant appearing implicitly in the above inequality does not depend on $p$. Thus, by applying Gr\"onwall's inequality
\[ \norm{\nabla \theta_{S}}^{p}_{L^{p}}+\norm{ \nabla u_{T}}^{p}_{L^{p}} \lesssim \left(\norm{\nabla \theta^{0}_{S}}^{p}_{L^{p}}+\norm{\nabla u^{0}_{T}}^{p}_{L^{p}}\right)\text{ exp} \left( p\int_{0}^{t} (\norm{\nabla u_{S}(\tau)}_{L^{\infty}}+1)\diff \tau\right),\]
which gives
\begin{equation}\label{eq:before:limit}
\norm{\nabla \theta_{S}}_{L^{p}}+\norm{ \nabla u_{T}}_{L^{p}} \lesssim \left(\norm{\nabla \theta^{0}_{S}}_{L^{p}}+\norm{\nabla u^{0}_{T}}_{L^{p}}\right)\text{ exp} \left( \int_{0}^{t} (\norm{\nabla u_{S}(\tau)}_{L^{\infty}}+1)\diff \tau\right).
\end{equation}
Finally, by taking limits in \eqref{eq:before:limit}, we get
\begin{align} 
\norm{\nabla \theta_{S}}_{L^{\infty}} + \norm{\nabla u_{T}}_{L^{\infty}} &= \displaystyle\lim_{p\to\infty} (\norm{\nabla \theta_{S}}_{L^{p}}+\norm{\nabla u_{T}}_{L^{p}}) \nonumber \\
 \label{BKM2}&\lesssim \left(\norm{\nabla \theta^{0}_{S}}_{L^{\infty}}+\norm{\nabla u^{0}_{T}}_{L^{\infty}}\right)\text{ exp} \left( \int_{0}^{t} (\norm{\nabla u_{S}(\tau)}_{L^{\infty}}+1)\diff \tau\right). 
\end{align} \\
\textit{Step 4: Final stage and blow-up criterion.} To conclude the proof we just need to collect estimates \eqref{BKM1} and \eqref{BKM2} to notice that setting $E(t)=1+\norm{u_{S}}^2_{{H}^s}+\norm{u_{T}}^2_{{H}^s}+\norm{\theta_{S}}^2_{{H}^s}$, we have that
\begin{equation*}
\dot{E}(t) \lesssim (\norm{\nabla \theta^{0}_{S}}_{L^{\infty}}+\norm{\nabla u^0_{T}}_{L^{\infty}}) \text{ exp} \left(\int_{0}^{t} (\norm{\nabla u_S(\tau)}_{L^{\infty}}+1)\diff \tau\right)E(t)+(\norm{u_{S}}_{W^{1,\infty}}+1)\ E(t).
\end{equation*}
By the Sobolev embedding there exists a constant $C>0$ such that $\norm{\nabla u_{S}}_{L^{\infty}} \leq C E(t)$.  Consequently, using this fact and Gr\"onwall's inequality, we obtain the equivalence between \eqref{eq:X(t)toinfty} and \eqref{eq:utoinfty}. \\
\begin{remark}
We have derived the blow-up criterion theorem (Theorem \ref{blowupcriteria}) implicitly assuming that $u_{S},u_{T},\theta_{S} \in C([0,T];H^{s+1})$, although it is only guaranteed that $u_{S},u_{T},\theta_{S} \in C([0,T];H^{s})$ by Theorem \ref{mainth}. The estimates can be made rigorous via standard approximation procedure and a routinary convergence argument. We shall omit this part to avoid redundancy.
\end{remark}
\end{proof}

\section{Conclusions}\label{5}
In this paper we have established the local well-posedness of the Incompressible Slice Model (ISM) in the Sobolev space $H^{s}(\Omega)$, with $s>2$ being an integer. This model is intended for the study and simulation of atmospheric fronts, and numerical approximations of it are being used for this purpose, supporting the evidence that fronts can develop from certain initial conditions. Our result provides a first step for addressing other possible questions such as stability/instability phenomena, which we do by providing a characterisation of a broad class of equilibrium solutions of the ISM, and deriving formal and nonlinear stability conditions around them. We have also constructed a blow-up criterion, based on an analysis of the
$L^{p}$ norms of the gradients of $u_T$ and $\theta_S$. This criterion is useful for numerical simulations, since one can track the evolution of the $L^{\infty}$ norm of the gradient of the velocity $u_S$ to see whether solutions are likely to blow up in finite time. 

Notice that we have worked with the most general case $\Omega\subset\mathbb{R}^{2}$ an open domain with smooth boundary, since in the real world flows often move in bounded domains with constraints coming from the boundaries. Moreover, initial boundary value problems manifest quite a particular behavior and richer phenomena occur in contrast with a periodic domain $\mathbb{T}^{2}$ or the real plane $\mathbb{R}^{2}$. Therefore, studying the problem with this generality entails certain extra difficulties we have had to overcome.

\subsection{Outlook for further research}
There still remain many open questions which are left for future research. We would like to present some of these directions here:  \\ \\
\textit{Global existence of smooth solutions vs blow-up in the Incompressible Slice Model.} The problem of finite time blow-up formation for solutions of the Incompressible Slice Model is an intriguing open problem. Since this system of equations bears a resemblance to the 2D inviscid Boussinesq equation and therefore to the 3D axisymmetric Euler equation, this would be a major breakthrough.
\\ \\ 
\textit{Dissipative version of the Incompressible Slice Model.} It would be natural to study the Incompressible Slice Model with added viscosity, namely
\begin{align*}
  \partial_t u_S + u_S \cdot \nabla u_S 
-u_T{\hat{x}}
&= -\nabla p
 + \theta_S \hat{z} + \nu\Delta u_S,  \\
 \partial_t u_T + u_S \cdot\nabla u_T
+ u_S \cdot\hat{{x}}
&=   -z,  \\
\partial_t \theta_S + u_S \cdot\nabla\theta_S + u_T  &= 0,\\
\nabla\cdot u_S &=  0, 
\end{align*}
where $\nu>0$ is a positive constant. \\ \\
\textit{Stochastic version of the Incompressible Slice Model.}
An approach for including stochastic processes as multiplicative cylindrical noise in systems of PDEs was proposed in \cite{Principal}. Introduction of stochasticity in fluid systems can help account for two main factors:
\begin{enumerate}[1.]
\item The effect of small scale processes. \\
\item  The uncertainty coming from the numerical methods used to resolve these equations. 
\end{enumerate}
Recently, there have been several works studying the well-posedness of fluid dynamics equations with multiplicative cylindrical noise \cite{stoEuler2017,stoEuler2018,stochBouss2018,stoBurgers2018}. This approach can be also applied to the Incompressible Slice Model, yielding a new stochastic systems of equations
\begin{align*}
 \diff \omega_S + \mathcal{L}_{u_S} \omega_S \ \diff t+\displaystyle\sum_{i=1}^{\infty}  \mathcal{L}_{\xi_{i}} \omega_S \circ  \diff B^{i}_{t}&= \partial_{x} \theta_S \ \diff t +\partial_{z} u_T \ \diff t,  \\ 
  \diff u_T + \mathcal{L}_{u_S} u_T \ \diff t + u_S \cdot \hat{x} \ \diff t + \displaystyle\sum_{i=1}^{\infty}  \mathcal{L}_{\xi_{i}} u_T\circ \diff B^{i}_{t}&= -z \ \diff t, \\
 \diff \theta_S + \mathcal{L}_{u_S} \theta_S \ \diff t + \displaystyle\sum_{i=1}^{\infty}  \mathcal{L}_{\xi_{i}}\theta_S \circ \diff B^{i}_{t} + u_T \ \diff t &= 0, \\
 \nabla \cdot u_S = 0, \quad \nabla \cdot \xi_i&=0, \hspace{0.2cm} i=1,\ldots, \infty,
\end{align*}
where $\circ \diff B^i_t$ is understood as integration with respect to Brownian motion in the Stratonovich sense, and $\omega_S=\nabla^{\perp}\cdot u_S$ is the slice component of the vorticity. From this point, well-posedness of the stochastic problem can be investigated.

\section*{Acknowledgements} We are deeply indebted to Darryl Holm for many useful suggestions which have significantly improved this manuscript. The first author is has been partially supported by the grant  MTM2017-83496-P from the Spanish Ministry of Economy and Competitiveness and through the Severo Ochoa Programme for Centres of Excellence in R\&D€ (SEV-2015-0554). The second author has been supported by the Mathematics of Planet Earth Centre of Doctoral Training.

\newcommand{\etalchar}[1]{$^{#1}$}

\end{document}